\documentclass[11pt]{amsart}
\usepackage{amsmath}
\usepackage{amssymb}
\usepackage{amsthm}
\usepackage{graphicx} 
\usepackage{graphpap}
\usepackage{url}

\theoremstyle{plain}

\theoremstyle{plain}
\newtheorem{theorem}{Theorem}[section]
\newtheorem{proposition}[theorem]{Proposition}

\newtheorem{corollary}[theorem]{Corollary}

\theoremstyle{definition}
\newtheorem{lemmadefinition}[theorem]{Lemma-Definition}
\newtheorem{definition}[theorem]{Definition}

\newtheorem{question}[theorem]{Question}
\newtheorem{remark}[theorem]{Remark}

\newcommand{\bord}{\partial}

\newcommand{\C}{\mathbb{C}}

\newcommand{\fn}{f^\bullet_\gamma}
\newcommand{\fb}{f^\circ_\gamma}

\renewcommand{\ge}{\geqslant}

\newcommand{\Hy}{\mathbb{H}}
\newcommand{\HH}{\mathrm{H}}

\newcommand{\N}{\mathbb{N}}

\newcommand{\OO}{\mathcal{O}}

\newcommand{\Q}{\mathbb{Q}}

\newcommand{\R}{\mathbb{R}}

\newcommand{\Sn}{S^\bullet_\gamma}
\newcommand{\Sb}{S^\circ_\gamma}
\newcommand{\SLZ}{\mathrm{SL_2(\Z)}}
\newcommand{\Sph}{\mathbb{S}}

\newcommand{\term}[1]{{\bf #1}}

\newcommand{\U}{\mathrm{T}^1}

\newcommand{\Z}{\mathbb{Z}}

\begin{document}

\title{Divide monodromies and antitwists on surfaces}
\author{Pierre Dehornoy}
\thanks{PD is supported by the projects ANR-15-IDEX-02 and ANR-11LABX-0025-01}
\author{Livio Liechti}
\AtEndDocument{\bigskip{\footnotesize%
  \textsc{Pierre Dehornoy}\par {Univ. Grenoble Alpes, CNRS, Institut Fourier, F-38000 Grenoble, France} \par  
  \url{pierre.dehornoy@univ-grenoble-alpes.fr} \par
  \url{http://www-fourier.ujf-grenoble.fr/~dehornop/}\par
}}
\AtEndDocument{\bigskip{\footnotesize
  \textsc{Livio Liechti}\par {D\'epartement de Math\'ematiques, Universit\'e de Fribourg, Chemin du Mus\'ee 23, 1700 Fribourg, Suisse} \par  
  \url{livio.liechti@unifr.ch} \par
  \url{https://homeweb.unifr.ch/liechtli/pub/} \par
}}
\date{\today}

\begin{abstract}
A divide on an orientable 2-orbifold gives rise to a fibration of the unit tangent bundle to the orbifold.
We characterize the corresponding monodromies as exactly the products of a left-veering horizontal and a right-veering vertical antitwist 
with respect to a cylinder decomposition, 
where the notion of an antitwist is an orientation-reversing analogue of a multitwist. 
Many divide monodromies are pseudo-Anosov and we give plenty of examples.
In particular, we show that there exist divide monodromies with stretch factor arbitrarily close to one,  
and give an example none of whose powers can be obtained by Penner's or Thurston's construction of pseudo-Anosov mapping classes.
As a side product, we also get a new combinatorial construction of pseudo-Anosov mapping classes in terms of 
products of antitwists.
\end{abstract}


\maketitle


\section{Introduction}

The unit tangent bundle to an orientable 2-orbifold fibres over the unit circle in ways that can be encoded by specific  
multicurves on the orbifold so that the complementary regions admit a black-and-white coloring. Given such a multicurve, a 
fibre surface is defined by all the unit vectors based at points of the multicurve and pointing toward the black faces. 
Using the geodesic flow, this construction was discovered in 1917 by Birkhoff for the case where the multicurve is geodesic on a Riemannian surface~\cite{Birkhoff}. 
It was also used in~1975 by A'Campo to describe the monodromy map associated to a singularity of a complex algebraic curve using
relative arcs on the unit disc, ``partages'' in the original french paper~\cite{ACampo1, ACampo2}. A'Campo's construction was later generalized  
by Ishikawa, who replaced the disc by arbitrary surfaces, with or without boundary~\cite{Ishikawa}. 
In the context of Birkhoff's construction, it was noticed by Fried that the monodromy maps are of pseudo-Anosov type if the underlying surface is hyperbolic~\cite{Fried}.
We call the multicurves in this construction \term{divides} and the associated maps \term{divide monodromies}, and refer to Section~\ref{S:defi} for the precise definitions.

A'Campo and Ishikawa described divide monodromies as a product of three multitwists. 
Our first result is a structure theorem for divide monodromies, where we reduce the number of multitwists from three to two, 
but we actually replace the notion of multitwist with the new notion of \term{antitwist} --- a variant that reverses the orientation of the surface. 
These antitwists cannot be made along arbitrary curves: a specific combinatorial property is required, 
and we call surfaces and curves that satisfy them~\term{Ba'cfi-tiled surfaces}, for Birkhof-A'Campo-Fried-Ishikawa. 
We again refer to Section~\ref{S:defi} for the precise definitions. 

\begin{theorem}
\label{T:antitwist}
A divide monodromy is the product of a right-veering and a left-veering antitwist.
\end{theorem}

The statement of Theorem~\ref{T:antitwist} is illustrated in Figure~\ref{F:Intro} below. 
Depicted on the left is a genus 2 orbifold~$\OO$ with 2 cone points, and a divide~$\gamma$ in red on~$\OO$. 
The fibre~$\Sn$ sits in the unit tangent bundle~$\U\OO$. It consists of those tangent vectors based on~$\gamma$ and pointing toward the black faces. 
Depicted on the top right is a square-tiled model of~$\Sn$, where each square corresponds to an edge between two double points of~$\gamma$. The identifications of the sides are marked with colors. Note that travelling along the $(2,2)$-vector always brings one back to the initial position: this is the Ba'cfi-tiling property. The surface has two horizontal cylinders that correspond to the two black faces of~$\OO\setminus\gamma$, and three vertical cylinders that correspond to the three white faces of~$\OO\setminus\gamma$. The effect of a horizontal antitwist on one vertical cylinder made of four square tiles is depicted on the bottom right. It acts on every square tile like the composition of a transvection and a reflection along a horizontal axis. On the top horizontal cylinder, the images of the squares wrap almost twice along the cylinder, illustrating the fact that the corresponding black face contains an order 2 point.

\begin{figure}[h!]
	\includegraphics[width=.49\textwidth]{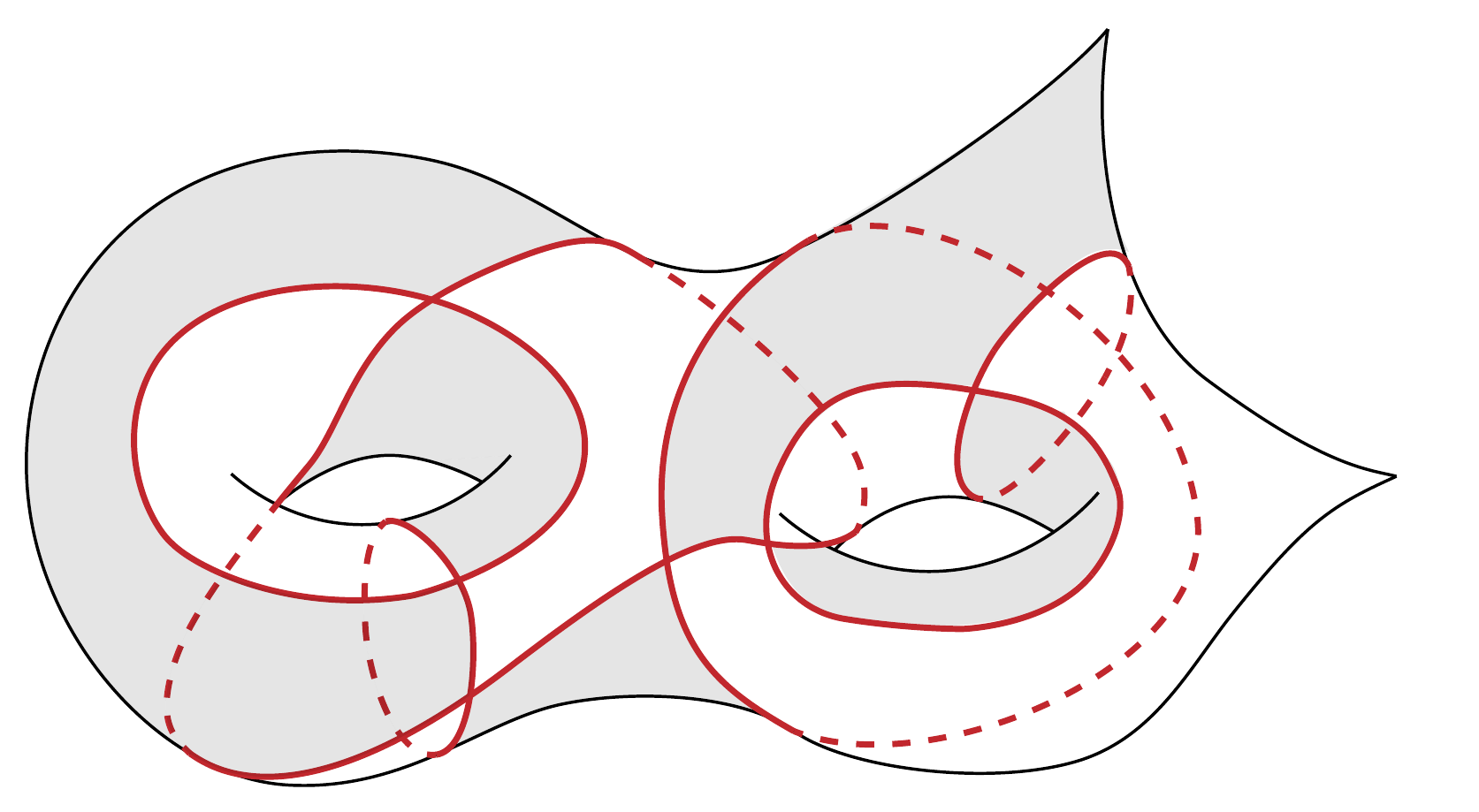}
	\includegraphics[width=.49\textwidth]{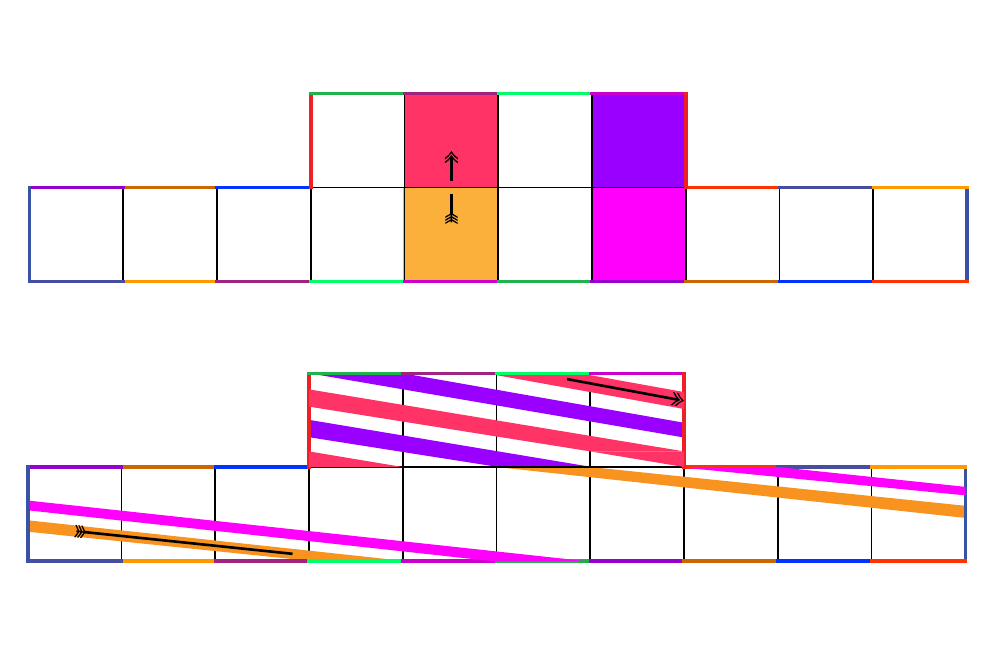}
	\caption{A divide on a genus 2 orbifold~$\OO$ with 2 cone points, and the corresponding Ba'cfi-tiled surface.}
	\label{F:Intro}
\end{figure}

We also prove the following converse result to Theorem~\ref{T:antitwist}. 

\begin{theorem}\label{T:converse}
Let~$S$ be a Ba'cfi-tiled surface and let~$f:S\to S$ be a composition of a left-veering horizontal and a right-veering vertical antitwist. 
Then there exists an orientable 2-orbifold~$\OO$ and a divide~$\gamma$ on~$\OO$ so that the associated fibre~$\Sn$ is homeomorphic 
to~$S$ and the divide monodromy is conjugate to~$f$.
\end{theorem}

A homeomorphism~$f$ of a surface is \term{pseudo-Anosov} if there exists a pair of transverse, 
singular measured foliations so that~$f$ stretches one of them by a factor~$\lambda>1$ and contracts the other by a factor~$1/\lambda$. 
By Thurston's classification of surface mapping classes, every irreducible mapping class either has a periodic or a pseudo-Anosov representative~\cite{ThurstonConstruction}. 
The number~$\lambda$ is called the \term{stretch factor} of the pseudo-Anosov mapping class. 
A particular feature of divide monodromies is that we can find among them pseudo-Anosov maps of minimal stretch factor for the given surface: 
both the CAT-map on the torus and the stretch factor minimizing pseudo-Anosov on the genus~2 surface are divide monodromies. 
We provide these examples and the calculation of the stretch factors in Section~\ref{S:mono}. 
The fact that these minimal stretch factor examples appear as divide monodromies is partially explained by 
the following result. 

\begin{theorem}\label{T:stretchfactor}
Let~$S$ be a Ba'cfi-tiled surface and let~$f:S\to S$ be a pseudo-Anosov divide monodromy. 
Denote by~$S^o$ the surface obtained from~$S$ by removing the vertices of the square tiles and by~$f^o$ the induced map. 
Then~$f^o$ minimizes the normalized stretch factor in its fibred cone of the Thurston norm. 
\end{theorem}

We can also use the geometric properties of~$\gamma$ for deriving properties of the divide monodromy. 
In this direction, the following result gives an upper bound on the stretch factor in terms of  the maximal diameter of the regions in the complement of~$\gamma$ in the orbifold. 

\begin{theorem}\label{T:diameter}
Given a hyperbolic good 2-orbifold~$\OO$ and a geodesic divide~$\gamma$ on~$\OO$ such that all regions of~$\OO\setminus\gamma$ have diameter at most~$D$, the stretch factor of the associated divide monodromy is bounded by~$e^{2D}$.
\end{theorem}

This theorem in particular implies the existence of divide monodromies whose stretch factors tend to~1, compare with Corollary~\ref{T:comparaison}.
In turn, this proves the existence infinitely many examples of divide monodromies not obtained by Penner's or Thurston's constructions of 
pseudo-Anosov mapping classes, see Corollary~\ref{T:noThurstonPenner}. This question of comparison is very natural. Indeed, A'Campo's description of divide monodromies 
as a product of three multitwists and our description of divide monodromies as a product of two antitwists is reminiscent of Thurston's construction of 
pseudo-Anosov mapping classes. Using algebraic criteria due to Shin and Strenner~\cite{SS} and Hubert and Lanneau~\cite{HL}, we prove the following stronger result.

\begin{theorem}\label{T:noPowerThurstonPenner}
There exists a divide monodromy~$f$ so that no power~$f^k$ is obtained from Thurston's or Penner's construction, where~$k\ge1$.
\end{theorem}
 
This result seems notable from the point of view of pseudo-Anosov theory, as there exist only few explicit combinatorial constructions of pseudo-Anosov mapping classes,
and the constructions of Penner and Thurston are possibly the most accessible ones. 
By exhibiting an explicit train track in a special case of Ba'cfi-tiled surfaces, we can give a new combinatorial construction of pseudo-Anosov maps in terms of 
a product of two antitwists. 

\begin{theorem}[Ba'cfi-construction of pseudo-Anosov mapping classes]
\label{Ba'cfi-pa_thm}
Let~$S$ be a Ba'cfi-tiled surface so that all cylinders have width at least 5. Furthermore, let~$\tau_h$ a left-veering horizontal antitwist and let~$\tau_v$ be a right-veering vertical antitwist. Then~$\tau_v\circ\tau_h$ is pseudo-Anosov with stretch factor at least~$\frac{5}{2}$.
\end{theorem}

The rest of this paper is organized as follows: 
Sections 2 to 5  are devoted to the definitions and to the proof of Theorems~\ref{T:antitwist} and~\ref{T:converse}. 
Section~\ref{S:stretch factor} contains our results on the stretch factors of pseudo-Anosov divide monodromies and the comparison with the constructions 
of Penner and Thurston. 
Section~\ref{S:TrainTrack} describes an {invariant train track} for the divide monodromy
 in the restricted situation where all the complementary regions have at least 5 sides. This leads to our Ba'cfi-construction of pseudo-Anosov mapping classes.
Finally, Section~\ref{S:Questions} contains some open questions on divide monodromies.

\medskip 
\noindent{\bf Acknowledments.} We thank Mario Shannon for suggesting Theorem~\ref{T:converse} and Erwan Lanneau for several discussions.



\section{Definitions}\label{S:defi}

\subsection{Divide surfaces in the unit tangent bundle}
First we recall the construction of Birkhoff-A'Campo-Fried-Ishikawa of specific surfaces in the unit tangent bundle to a surface~\cite[p.277]{Birkhoff}.

It starts with a compact Riemannian surface with empty boundary. 
(Actually, A'Campo works with the disc which has boundary, but here it is easier to assume the boundary being empty.)
It can be generalized to a {Riemannian 2-dimensional orbifold}, that is, a compact Riemannian surface with a finite number of {cone points}~$c_1, \dots, c_k$ where the total angle is~$2\pi/p_i$ for some integers $p_i$ called the orders of the points~$c_i$.
When the orbifolds are of the type~$\Hy^2/\Gamma$ for some Fuchsian group~$\Gamma$, they are called {hyperbolic good orbifolds}~\cite{ThurstonNotes}. 
In this case the unit tangent bundle~$\U\OO$ can be identified with~$\U\Hy^2/\Gamma$.

Given an orbifold~$\OO$, a \term{divide} on~$\OO$ is a finite collection of curves in general position such that the complement~$\OO\setminus\gamma$ is the union of discs that may be black-and-white colored. 
We always assume that the complement is indeed black-and-white colored.
Also $\gamma$ is not allowed to pass through cone points, except for cone points of order~$2$ which~$\gamma$ is allowed to visit exactly twice per cone point (in this case both branches of ~$\gamma$ make a U-turn at the cone point). 

\begin{definition}\label{D:dividesurface}
Given an orientable Riemannian 2-orbifold~$\OO$ and a divide~$\gamma$ on~$\OO$, the \term{divide link}~$\U\gamma$ is the set of unit vectors tangent to~$\gamma$. 
The \term{divide surface}~$\Sn$ is the subset of~$\U\OO$ consisting of those tangent vectors based on~$\gamma$ and pointing toward the black regions, plus all tangent vectors based at the double points of~$\gamma$. 
\end{definition}

\begin{remark}The Riemannian structure on~$\OO$ is used for defining the unit tangent bundle. But the description of the surface~$\Sn$ and the monodromy map~$f_\gamma$ we give later shows that they only depend on the topological type of~$\OO$ and the relative position of~$\gamma$ in~$\OO$ with respect to the cone points.
\end{remark}

The boundary~$\partial\Sn$ is the link~$\U\gamma$. 
Usually, one smooths~$\Sn$ around the fibres of the double points by pushing it a little forward for those vectors pointing toward black regions and by pulling it back a little for those vectors pointing toward the white regions. 

\begin{figure}[h]
\begin{center}
\includegraphics[width=.5\textwidth]{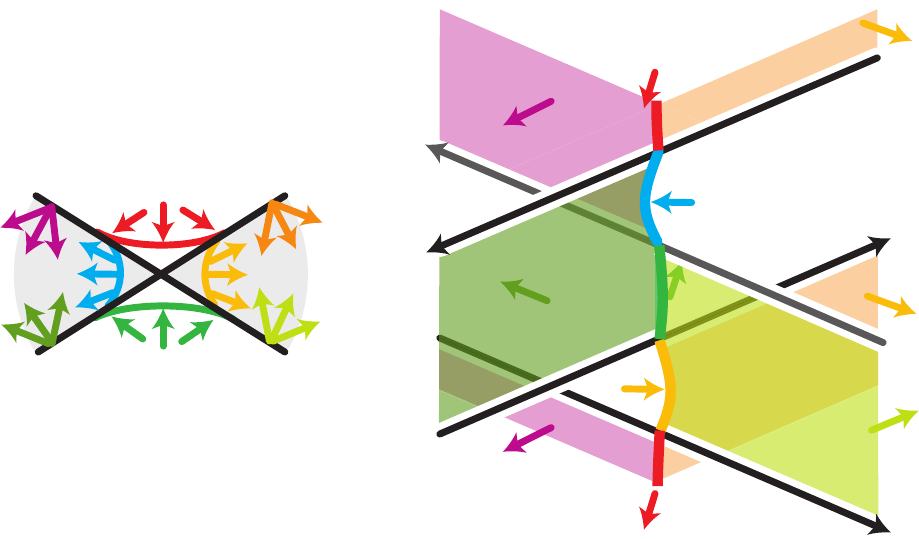}
\end{center}
\caption{The smoothed surface~$\Sn$ around a double point of the divide. Observe that the left picture is invariant by an order 2 rotation, and the right picture is invariant by an order 2 screw motion.}
\label{smoothedfibre}
\end{figure}

If $\gamma$ has $c$ components avoiding the cone points of order 2, and $d$ components visiting some cone point of order 2, then $\U\gamma$ has $2c+d$ components, so $\Sn$ is a surface with $2c+d$ boundary components.

The link~$\U\gamma$ is always fibred~\cite{Ishikawa} with possible fibre~$\Sn$. 
We say \emph{possible} since the fibration is in general not unique. 
However, the fibration that has~$\Sn$ as a fibre is the only one which we consider here.

\begin{definition}\label{D:dividemap}
The {monodromy map}~$f_\gamma:\Sn\to \Sn$ is called the \term{divide monodromy} associated to~$\gamma$.
\end{definition}

If the divide $\gamma$ is geodesic, then the surface~$\Sn$ is a Birkhoff section for the geodesic flow in~$\U\OO$. 
When~$\OO$ is negatively curved, the geodesic flow on~$\U\OO$ is of Anosov type. 
Fried remarked~\cite{Fried} that the weak foliations~$\mathcal F^{u}$ and $\mathcal F^{s}$ are transverse to~$S_\gamma$, 
and hence induce foliations~$\mathcal F^{u}\cap \Sn$ and~$\mathcal F^{s}\cap \Sn$ of the interior of~$\Sn$ that are~$f_\gamma$-invariant. 
Thus, in this case, $f_\gamma$ is a pseudo-Anosov map, with singularities on the boundary only. 

\subsection{Square tiles and Ba'cfi-tiled surfaces}
We now give a combinatorial model of divide monodromies.

A \term{square-tiled surface} is a surface~$S$ obtained from a finite number of copies of the unit square whose sides are identified by translations in the plane. 
We denote by~$Q(S)$ the set of squares composing~$S$.
Given a square $q$ in $Q(S)$, it makes sense to speak of the four cardinal directions, so that there are $N, S, W,$ and $E$\term{-neighbors}.
One can then describe a sequence of adjacent squares by giving a starting point and a sequence of directions. 
For example, we write $NE(q)$ for the north neighbor of the east neighbor of $q$.

\begin{definition}\label{D:Ba'cfitiled}
A square-tiled surface~$S$ is \term{Ba'cfi-tiled} if, for every $q$ in~$Q(S)$, one has $N\!EN\!E(q)=q$.
\end{definition}

The condition is equivalent to $EN\!EN(q)=q$ for every square~$q$, or to $SW\!SW(q)=q$, or to $W\!SW\!S(q)=q$.

A \term{horizontal cylinder} in a square-tiled surface~$S$ is a set of the form $q, E(q), \dots E^{w(q)-1}(q)$ with $E^{w(q)}(q)=q$; it is the union of a finite number $w(q)$ of squares. 
We denote by $\gamma_h(q)$ the core of the horizontal cylinder through~$q$, that is, the image of the horizontal line passing through the center of~$q$.
One similarly defines \term{vertical cylinders} and the curves $\gamma_v(q)$.
For a square-tiled surface~$S$ and a square~$q$ in~$S$, it makes sense to speak of its $N, S, W,$ and $E$-sides, denoted by~$e^N(q), e^S(q), e^W(q),$ and $e^E(q)$, respectively. 

We then define $\tau_h(q)$ as the map sending $e^S(q)$ on $e^N(E(q))$, $e^N(q)$ on $e^S(W(q))$, and the vertical sides of $q$ on segments of slope $(w(q)-2,-1)$ connecting the extremities of $e^N(E(q))$ and~$e^S(W(q))$. 
Thus $\tau_h(q)$ is a parallelogram in~$S$ with two horizontal sides and two sides of slope~$(w(q)-2,-1)$.
Note that $\tau_h(q)$ is included in the horizontal cylinder associated to~$q$, see Figure~\ref{F:HorizontalAntitwist}.

For $n$ an integer, we denote by $\tau^{(n)}_h(q)$ the image of $\tau_h(q)$ under the $n{-}1$st power of a left-veering Dehn twist about~$\gamma_h(q)$. 
In Figure~\ref{F:Intro}, $\tau^{(2)}_h(q_i)$ is shaded for two square tiles~$q_i$ in the top horizontal cylinder, as well as $\tau^{(1)}_h(q_j)$ for two square tiles~$q_j$ in the bottom horizontal cylinder.

Since $\tau^{(n)}_h(q)$ does not have the same horizontal sides as~$q$, it is in general hard to apply such transformations to all squares of a square-tiled surface in a coherent way. 
However, for Ba'cfi-tiled surfaces this is possible.

\begin{lemmadefinition}\label{D:antitwist}(Compare with Figure~\ref{F:HorizontalAntitwist})
Let $S$ be a Ba'cfi-tiled surface, and let $\bar n:Q(S)\to\Z$ be a function that is constant on the horizontal cylinders. 
Then the parallellograms $\{\tau^{(\bar n(q))}_h(q)\}_{q\in Q(S)}$ tile~$S$. 
Therefore the maps~$\{\tau^{(\bar n(q))}_h\}_{q\in Q(S)}$ can be glued into a well-defined map~$\tau^{(\bar n)}_h:S\to S$.
Such a map is called a \term{horizontal antitwist}. 
If the function~$\bar n$ takes values in~$\Z_{\ge 1}$ one says that $\tau^{(\bar n)}_h$ is \term{left-veering}.
\end{lemmadefinition}

\begin{figure}[hbt]
\begin{picture}(73,27)(0,0)
\put(0,0){\includegraphics[width=.42\textwidth]{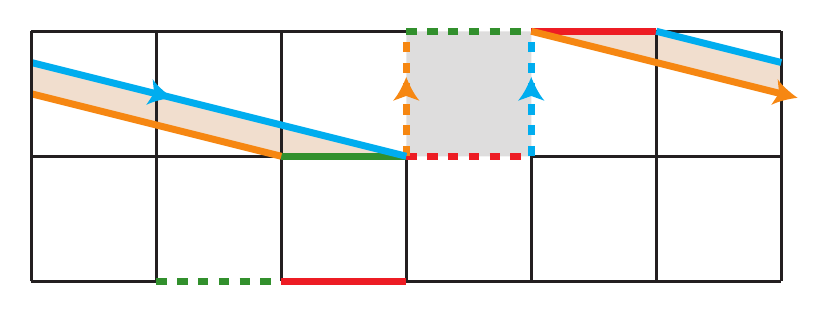}}
\put(7,-1){$3$}
\put(18,-1){$4$}
\put(29,-1){$5$}
\put(41,-1){$6$}
\put(52,-1){$1$}
\put(63,-1){$2$}
\put(7,26.5){$1$}
\put(18,26.5){$2$}
\put(29,26.5){$3$}
\put(41,26.5){$4$}
\put(52,26.5){$5$}
\put(63,26.5){$6$}
\end{picture}
\caption{A Ba'cfi-tiled surface with 12 square tiles. The identification of the horizontal sides is indicated by the numbers, and the vertical sides are identified by the obvious horizontal translation. 
A top-center square tile and its image under a left-veering antitwist~$\tau^{(1)}_h$ on the top horizontal cylinder are shown.}
\label{F:HorizontalAntitwist}
\end{figure}
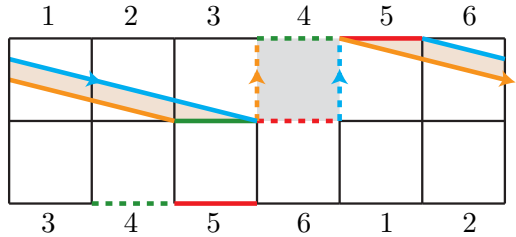

\begin{proof}
Let $C$ be the set of squares that compose an arbitrary horizontal cylinder. 
Since the function~$\bar n$ is constant on $C$, the parallellograms~ $\{\tau^{(\bar n(q))}_h(q)\}_{q\in C}$ tile~$C$. 
Moreover the Ba'cfi-character of~$S$ ensures that every horizontal side is sent on the same side by the maps induced by the two adjacent squares. 
Hence, the maps on all horizontal cylinders are coherent and the parallelograms~$\{\tau^{(\bar n(q))}_h(q)\}_{q\in Q(S)}$ tile~$S$.
\end{proof}

The term left-veering refers to the segment transverse to a horizontal cylinder being sent on a segment that is inclined to the left when looking from the boundary.
One similarly defines \term{vertical antitwists}. 
On Figure~\ref{F:VerticalAntitwist} we show an example of a \term{right-veering} vertical antitwist. 

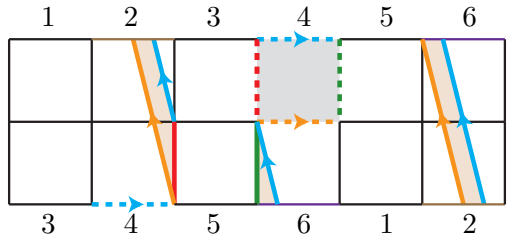
\begin{figure}[hbt]
\begin{picture}(73,27)(0,0)
\put(0,0){\includegraphics[width=.42\textwidth]{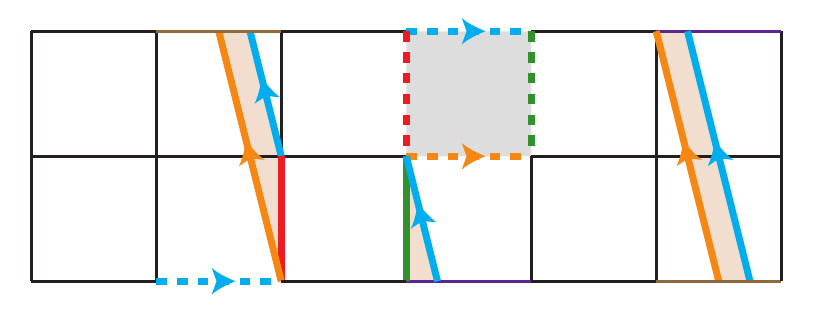}}
\put(7,-1){$3$}
\put(18,-1){$4$}
\put(29,-1){$5$}
\put(41,-1){$6$}
\put(52,-1){$1$}
\put(63,-1){$2$}
\put(7,26.5){$1$}
\put(18,26.5){$2$}
\put(29,26.5){$3$}
\put(41,26.5){$4$}
\put(52,26.5){$5$}
\put(63,26.5){$6$}
\end{picture}
\caption{The same Ba'cfi-tiled surface as in Figure~\ref{F:HorizontalAntitwist}, and the image of the top-center square tile under a right-veering antitwist~$\tau^{(1)}_v$ along the corresponding vertical cylinder.}
\label{F:VerticalAntitwist}
\end{figure}


\bigskip
\section{The surface~$\Sn$ as a square-tiled surface}

\subsection{The surface~$\Sn$ when $\gamma$ does not pass through an order~2 cone point.}
We start from an orientable 2-orbifold~$\OO$ and a geodesic divide~$\gamma$ avoiding the order 2 points of~$\OO$. 
We consider~$\gamma$ as a graph embedded in~$\OO$ whose vertices are the double points of~$\gamma$ and whose edges are simple arcs connecting double points. 
For every edge~$e$, there is an associated rectangular ribbon~$r^\bullet_e$ in~$\Sn$. 
The boundary of~$r^\bullet_e$ can be decomposed into six parts, namely 
\begin{itemize}
\item two horizontal sides that correspond to the two oriented lifts of~$e$ in~$\U\OO$,
\item two pairs of two vertical sides, each pair consisting of those tangent vectors at one extremity of~$e$ pointing towards an adjacent black and an adjacent white face, respectively, 
see Figure~\ref{buildsquares}.
\end{itemize}
If we contract the horizontal sides of~$\Sn$, we are left with a square that can be oriented so that its horizontal sides correspond to the tangent vectors based at double points and pointing toward the white faces, while the vertical sides correspond to the tangent vectors pointing toward the black faces. 

\begin{figure}[h]
\begin{center}
\includegraphics[width=.8\textwidth]{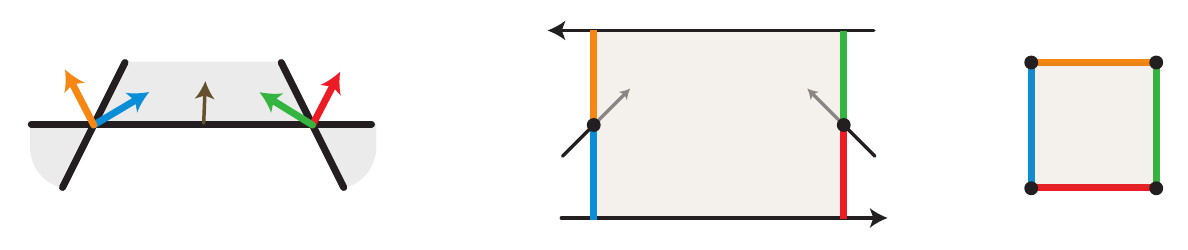}
\caption{Building a square-tiled model for the fibre~$\Sn$.}
\label{buildsquares}
\end{center}
\end{figure}

Denoting by~$d_\gamma$ the number of double points of~$\gamma$, the surface~$\Sn$ is made from~$2d_\gamma$ ribbons.
Considering them as squares, we can see~$\Sn$ as a square-tiled surface. 
Writing~$\sharp\gamma$ for the number of geodesics in~$\gamma$, we have $\vert\chi(\Sn)\vert=2d_\gamma$ and $\sharp\bord\Sn=2\sharp\gamma$, thus 
\begin{equation}
\label{Genre}
g(\Sn) = d_\gamma - \sharp\gamma+1.
\end{equation}

Since the vertical sides of the square tiles correspond to the vectors pointing toward the black faces, 
every horizontal cylinder in~$\Sn$ is obtained by picking a black face and following its sides along~$\gamma$.
Similarly, vertical cylinders are in one-to-one correspondence with white faces of~$\OO\setminus\gamma$.

The four sides adjacent to a double point are easily described: they are aligned along a diagonal of slope~$(1,1)$. 
The Ba'cfi-condition is then easily checked for~$\Sn$. 
Indeed, the path given by turning a tangent vector at a double point of~$\gamma$ 
by a full turn is closed and traverses the four adjacent sides.

The divide~$\gamma$ is in general not assumed to be geodesic. 
However, locally it is close to a geodesic divide. 
For such a divide the surface~$\Sn$ is canonically cooriented by the geodesic flow in~$\U\OO$. 
Pairing with the orientation of~$\U\OO$ we obtain a canonical orientation on~$\Sn$. 
We always give this orientation to~$\Sn$, even when~$\gamma$ is not geodesic. 
Also we will pictures in such a way that the geodesic flow goes from front to back of the screen.

\smallskip
\noindent{\bf Example 1:} The cross on a sphere with four cone points yields a Ba'cfi-tiled surface with four square tiles.
\begin{center}
\includegraphics[width=.5\textwidth]{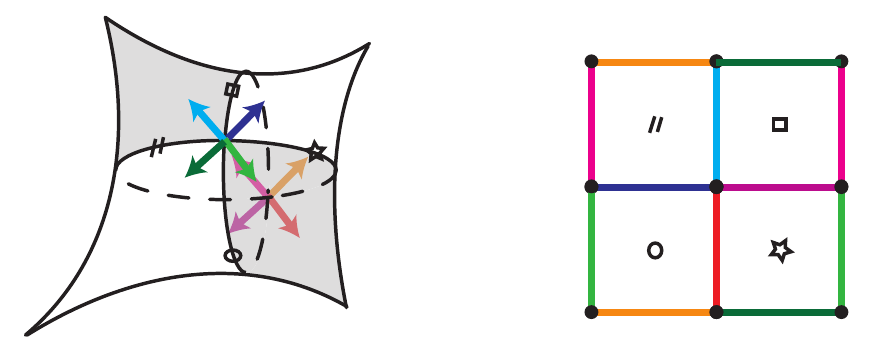}
\end{center}

\noindent{\bf Example 2:} Birkhoff-Fried's original surface. 
This surface was also studied by Ghys, Hashiguchi and Brunella, who computed and used the monodromy map~\cite{Ghys, Hashiguchi, Brunella}.

The numbers on the picture represent the vertical sides. 
In the figure on the right, they correspond to the place where the surface was smoothed, not to the direction at which the tangent vectors point. 
This makes a difference for the white regions, but not for the black ones. 
For example the vectors number 7 point toward NE. 
The good point of this convention is that we will not have to change the numbering when we describe the sister surface~$\Sb$ later on.

\begin{center}
\begin{picture}(150,80)(0,0)
\put(0,0){\includegraphics[width=.882\textwidth]{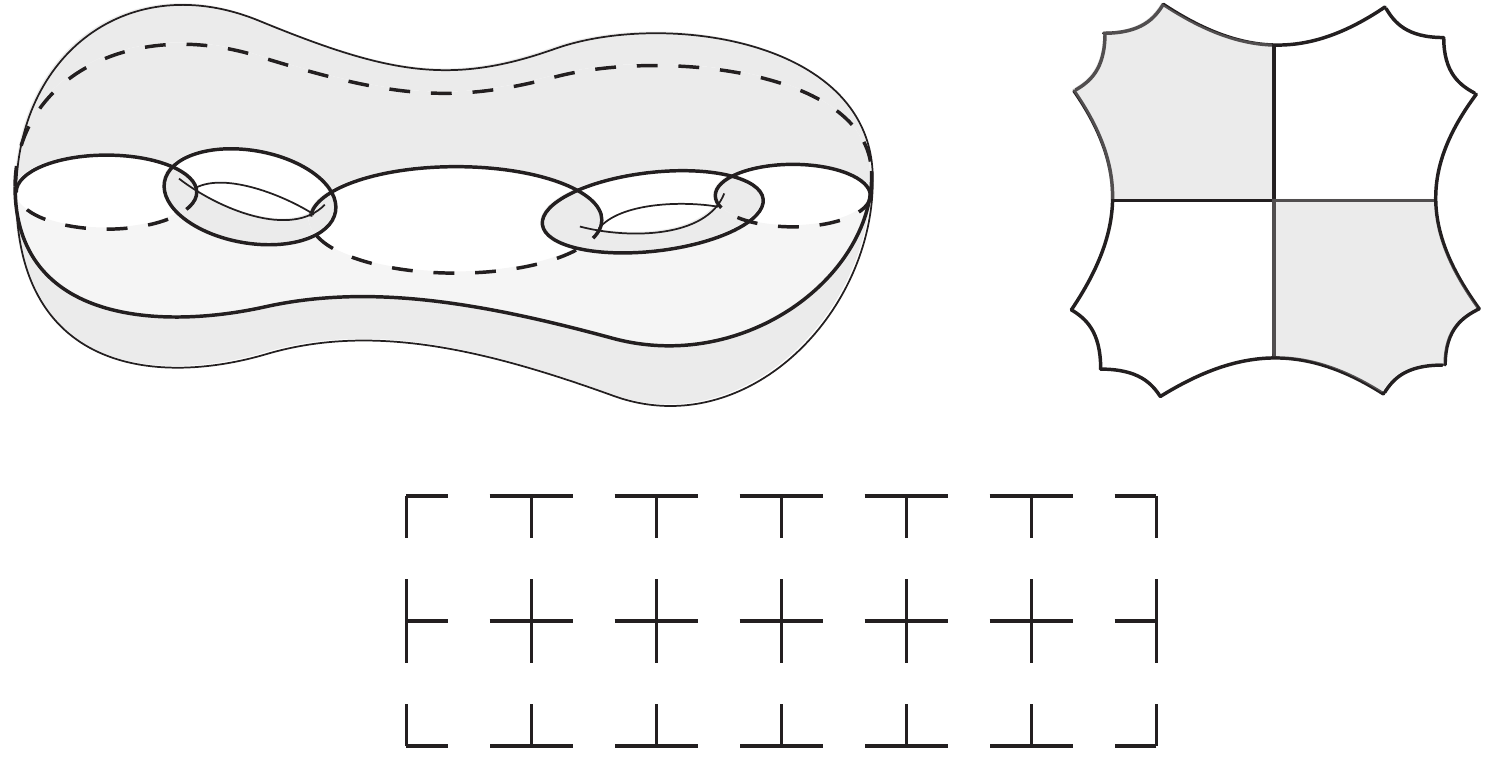}}
\put(119,57.5){$\alpha$}
\put(129.5,63){$\beta$}
\put(123,74.5){$\gamma$}
\put(123,37){$\gamma$}
\put(113,75){$\delta$}
\put(142,75){$\delta$}
\put(108,70){$\epsilon$}
\put(108,41){$\epsilon$}
\put(108.5,60){$\zeta$}
\put(147,60){$\zeta$}
\put(125.5,47){$\eta$}
\put(132,36){$\theta$}
\put(132,74){$\theta$}
\put(113,36){$\iota$}
\put(142,36){$\iota$}
\put(147,70){$\kappa$}
\put(147,41){$\kappa$}
\put(108,50){$\lambda$}
\put(147,50){$\lambda$}
\put(136,53){$\mu$}
\put(46.5,19.5){$\alpha$}
\put(59.3,19.5){$\beta$}
\put(72,19.5){$\gamma$}
\put(84.5,19.5){$\delta$}
\put(97.3,19.5){$\epsilon$}
\put(109.5,19.5){$\zeta$}
\put(46.5,6.5){$\eta$}
\put(59.3,6.5){$\theta$}
\put(72,6.5){$\iota$}
\put(84.5,6.5){$\kappa$}
\put(97.3,6.5){$\lambda$}
\put(109,6.5){$\mu$}
\put(46.5,25.3){$1$}
\put(59.3,25.3){$2$}
\put(71.6,25.3){$3$}
\put(84.2,25.3){$4$}
\put(97.5,25.3){$5$}
\put(109.5,25.3){$6$}
\put(46.5,0){$3$}
\put(59.1,0){$4$}
\put(71.6,0){$5$}
\put(84.5,0){$6$}
\put(98,0){$1$}
\put(109.5,0){$2$}
\put(46.5,13){$7$}
\put(59.3,13){$8$}
\put(71.6,13){$9$}
\put(83.1,13){$10$}
\put(95.8,13){$11$}
\put(108.2,13){$12$}
\put(39,19.3){$13$}
\put(115,19.3){$13$}
\put(51.8,19.3){$14$}
\put(64.5,19.3){$15$}
\put(77,19.3){$16$}
\put(89.8,19.3){$17$}
\put(102,19.3){$18$}
\put(39,6.5){$19$}
\put(115,6.5){$19$}
\put(51.8,6.5){$20$}
\put(64.5,6.5){$21$}
\put(77,6.5){$22$}
\put(89.8,6.5){$23$}
\put(102,6.5){$24$}
\put(140.5,57){$12$}
\put(126,53){$7$}
\put(129,69){$8$}
\put(116.2,38){$9$}
\put(141.5,70){$10$}
\put(109.5,45){$11$}
\put(129,57){$2$}
\put(126,41){$3$}
\put(138.5,72){$4$}
\put(111.2,40){$5$}
\put(144.5,65.3){$6$}
\put(113,53){$1$}
\put(112,57){$13$}
\put(124.5,57){$14$}
\put(124.5,69){$15$}
\put(116,72){$16$}
\put(112,70){$17$}
\put(110,65.3){$18$}
\put(128.5,53){$19$}
\put(128.5,41){$20$}
\put(137,38.4){$21$}
\put(142,40.2){$22$}
\put(144,44.7){$23$}
\put(140.6,53){$24$}
\end{picture}
\end{center}

\subsection{The sister surface~$\Sb$}
For building~$\Sn$ we made the arbitrary choice of considering those tangent vectors pointing toward black regions. 
It is equally natural to consider the other surface, which we denote by~$\Sb$, that is, the set of tangent vectors based on~$\gamma$ and pointing toward white regions. 
We smooth~$\Sb$ with the convention opposite to the one we used for~$\Sn$: we push a little bit the vectors based on double points of~$\gamma$ and pointing toward the white faces, and we pull backward those pointing toward the black faces, see Figure~\ref{smoothedfibre2}.

\begin{figure}[h]
\begin{center}
\includegraphics[width=.5\textwidth]{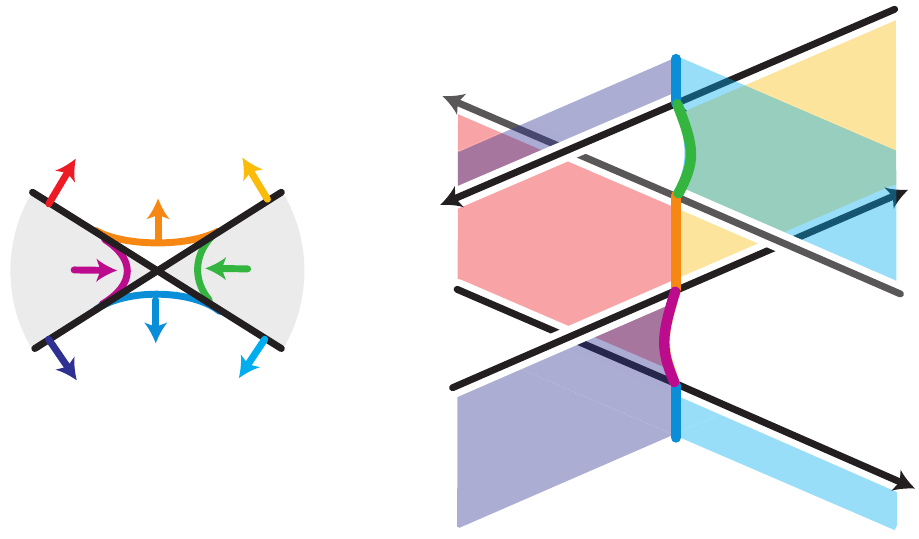}
\caption{The smoothed surface~$\Sb$ around a double point of the divide.}
\label{smoothedfibre2}
\end{center}
\end{figure}

There is a canonical homeomorphism~$\sigma:\Sn\to\Sb$ obtained by turning the vectors by~$180^\circ$, without changing their base points. 
This homeomorphism~$\sigma$ is the restriction to~$\Sn$ of the corresponding involution of~$\U\OO$. 
In the same way as~$\Sn$, the surface~$\Sb$ can be cut in square tiles, and~$\sigma$ canonically identifies the squares of both surfaces, 
sending horizontal sides on horizontal sides and vertical sides on vertical sides. 
When $\gamma$ is a geodesic divide, the geodesic flow on~$\U\OO$ coorients also~$\Sb$. 
Unfortunately,~$\sigma$ reverses the orientation, so that on the pictures~$\Sb$ and~$\Sn$ always have opposite orientations. 

The fibre~$\Sb$ in the Examples~1 and~2 are the same as the fibre~$\Sn$: 
thanks to our convention the representation as square-tiled surfaces of~$\Sn$ and~$\Sb$ coincide, up to reversing the orientation of the plane.

\subsection{Dealing with cone points of order 2}
Around a double point of~$\gamma$, the construction of~$\Sn$ and~$\Sb$ is invariant under rotation by~$180^\circ$. 

Now one  considers the case where~$\gamma$ goes through a conic point of order~2 exactly twice. 
Note on Figures~\ref{smoothedfibre} and~\ref{smoothedfibre2} that the construction of the surfaces~$\Sn$ and~$\Sb$ are invariant under the action of the order 2 screw motion, which correspond to the lift in~$\U\OO$ of the action of an order 2 rotation on the surface~$\Sigma$. 
Therefore the surfaces~$\Sn$ and~$\Sb$ survive after modding out by the order 2 screw motion, and this is how one defines the fibre surface in this case. 
Also, every arc of~$\U\gamma$ going through a double point is identified with its opposite, so that such arcs only account for~1 towards~$\sharp\gamma$. 
However, there are also only two ribbons adjacent to this double point, so the contribution towards~$d_\gamma$ is also~1, so that Formula~\eqref{Genre} for the genus still holds.

This case gives rise to interesting examples, in particular concerning the construction of pseudo-Anosov maps with small stretch factor.

\medskip
\noindent
\begin{minipage}{.4\textwidth}
{\bf Example 3:} The height on a sphere with three cone points of order 2, $q$, and~$r$, respectively. The surfaces~$\Sn$ and~$\Sb$ consist of one square only, with~$g(\Sn)=g(\Sb)=1$.
\end{minipage}
\begin{minipage}{.6\textwidth}
\begin{center}
\includegraphics[width=\textwidth]{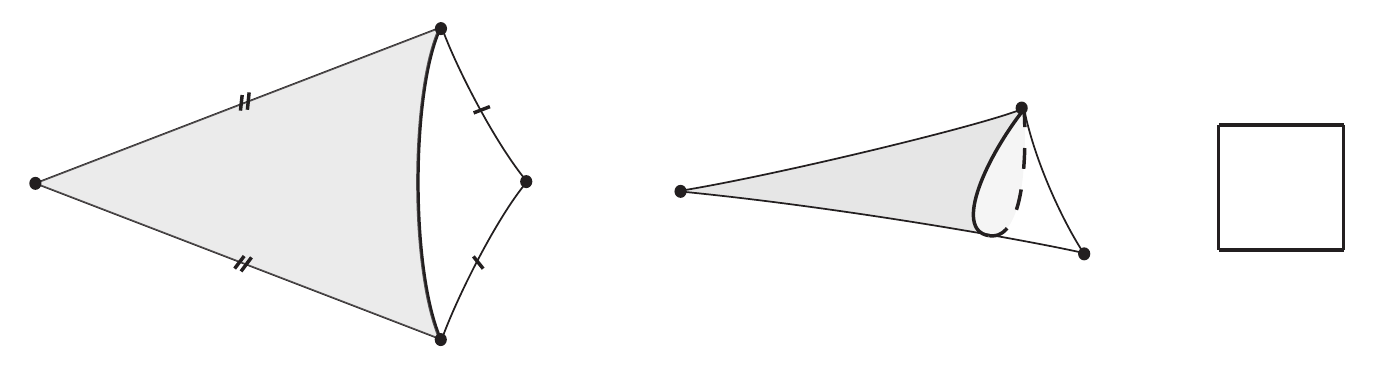}
\end{center}
\end{minipage}

\medskip
\noindent{\bf Example 4:} The loop, on the same orbifolds. In this case the surface~$\Sn$ is made from three square tiles, with~$g(\Sn)=2$. 
The fibre of the order 2 point and of the double point are represented with long and short dotted lines.
\begin{center}
\begin{picture}(130,45)(0,0)
\put(0,0){\includegraphics[width=.71\textwidth]{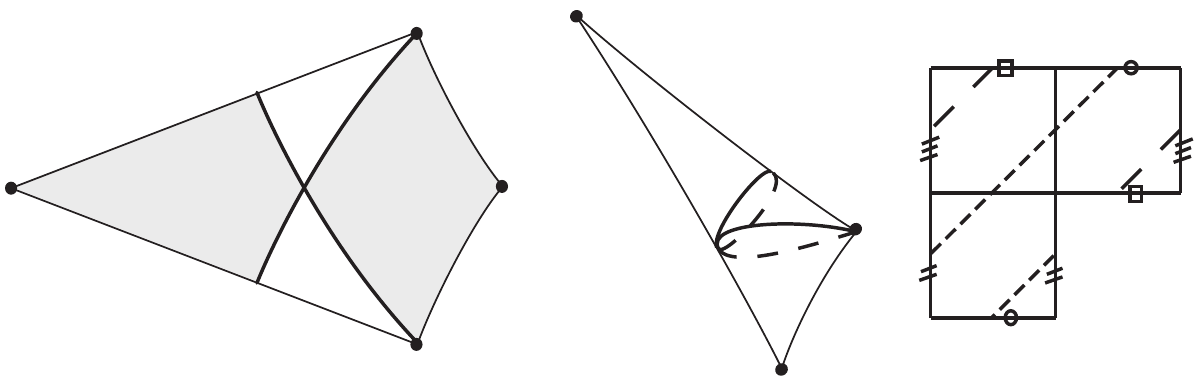}}
\put(27,25){$c$}
\put(27,17){$c$}
\put(38,27.5){$a$}
\put(38,14){$b$}
\put(104.5,27){$b$}
\put(104.5,13.5){$c$}
\put(118,27){$a$}
\end{picture}
\end{center}

\noindent{\bf Example 5:} 
The even ping-pong, with $4n-1$ squares and $g(\Sn)=2n$. 
The identifications of the vertical sides are the obvious ones, row by row. 
Every horizontal cylinder has 4 squares (except the first and last ones), which corresponds to the fact that every black region (except the first and last one) has 4 sides. 
The vertical cylinders are also made of 4 squares (except the last one), since the white faces also have 4 sides. 
The fibres of two double points are indicated by dotted lines.
\begin{center}
\begin{picture}(130,45)(0,0)
\put(0,0){\includegraphics[width=.75\textwidth]{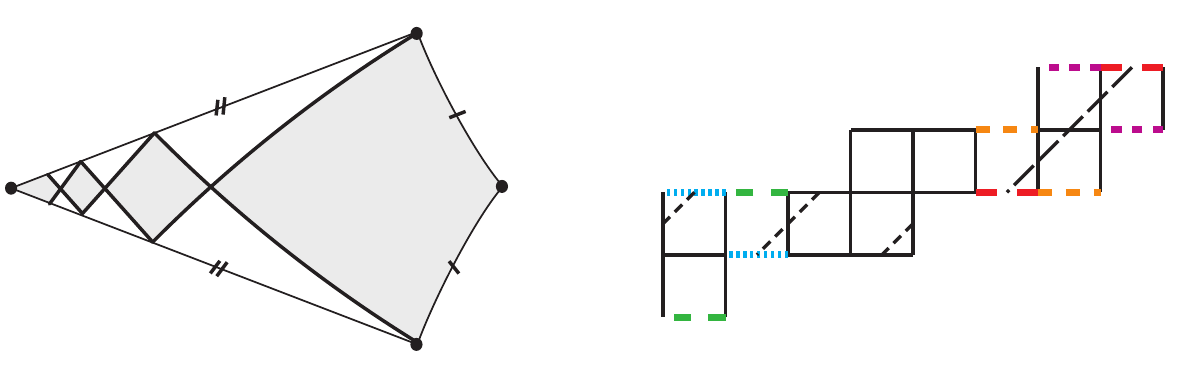}}
\put(34,29){$a$}
\put(34,13){$b$}
\put(19,17){$c$}
\put(19,25){$d$}
\put(14,23){$e$}
\put(14,18){$f$}
\put(10,18){$g$}
\put(10,23){$h$}
\put(7,24.5){$i$}
\put(7,16.5){$j$}
\put(4,17){$k$}
\put(119,30){$a$}
\put(113,30){$b$}
\put(113,23){$c$}
\put(106,23){$d$}
\put(99.5,23){$e$}
\put(92.5,23){$f$}
\put(92.5,16.5){$g$}
\put(86,16.5){$h$}
\put(79.5,16.5){$i$}
\put(73,16.5){$j$}
\put(73,10){$k$}
\end{picture}
\end{center}

\noindent{\bf Example 6:} 
The odd ping-pong, with $4n+1$ squares and $g(\Sn)=2n+1$. 
The picture of~$\gamma$ is the same with one more crossing, and the picture of~$\Sn$ is the same with two more square tiles.

\section{The monodromy map~$f_\gamma$ and the proof of Theorem~\ref{T:antitwist}}
\label{S:mono}

As soon as $\gamma$ is filling, the surface~$\Sn$ and $\Sb$ are fibre surfaces in~$\U\OO$, as proven by A'Campo (when $\OO$ is a disc) and by Ishikawa (when $\OO$ is a closed surface). 
Actually these authors do more, since they describe very explicitely the associated fibration of~$\U\OO\setminus\U\gamma$ over~$\Sph^1$. 
For that it is enough to consider a Morse function~$h_\gamma$ on~$\OO$ which has exactly one maximum in every black face of~$\OO\setminus\gamma$, one minimum in every white face, and which vanishes on~$\gamma$. 
One also introduces an auxilliary bump function~$\chi:\OO\to[0,1]$ such that for every critical point~$p$ of~$h$ the function $\chi$ is constant equal to~1 on a small disc centered at~$p$, and vanishes away from a neighborhood of the union of these discs. 
For $\eta$ small enough  the argument of the function~$H_{\gamma, \eta}:\U\OO\setminus\U\gamma\to\C$, $(x,u)\mapsto h_\gamma(x)+i\eta dh_\gamma(x)(u) - \frac 1 2 \eta^2\chi(x)d^2h_\gamma(x)(u,u)$  defines a fibration over the circle~\cite{ACampo2, Ishikawa}. 
The surfaces~$\Sn$ and $\Sb$ correspond to the preimages of~$i$ and $-i$ respectively.

In order to describe the monodromy maps, it is enough to consider a set of curves~$c_1, \dots, c_k$ that together cut~$\Sn$ into a collection of discs and to understand their images. 
For $i=1, \dots, k$, the image~$f_\gamma(c_i)$ is not uniquely defined, but its isotopy class is given by those curves in the boudary surface~$\Sn\times\{1\}$ of~$\U\OO\setminus\Sn$ that are isotopic to~$c_i\times\{0\}$. 

Here, instead of describing the monodromy maps~$\Sn\to\Sn$ and $\Sb\to\Sb$ directly, it is easier to compute the half-return maps~$\fn: \Sn\to\Sb$ and~$\fb: \Sb\to\Sn$, and then to compose them. 
So what we really do is to consider the arcs in~$\Sn$ that correspond to the edges of the square tiles, and to find arcs in~$\Sb$ to which they are isotopic through the correct connected component of~$\U\OO\setminus(\Sn\cup\Sb)$.

\subsection{The maps~$\fn$ and~$\fb$}
Concerning~$\fn$, we first observe that the horizontal cylinders are preserved. 
This is immediate from the fact that the set of vectors pointing into a black face is isotopic to the set of those vectors pointing outward of the same face. 

Then observe that the image under~$\fn$ of the horizontal sides is rather simple. 
Indeed a horizontal side corresponds to vectors pointing toward the white faces, as the red and the green arcs in Figure~\ref{smoothedfibre}. 
Since these vectors have been pulled backward when smoothing the surface~$\Sn$, these vectors are just pushed forward. 
Thus they are sent on the corresponding horizontal sides of~$\Sb$, see Figure~\ref{horizontalimage} below.

\begin{figure}[h]
\begin{center}
\includegraphics[width=.4\textwidth]{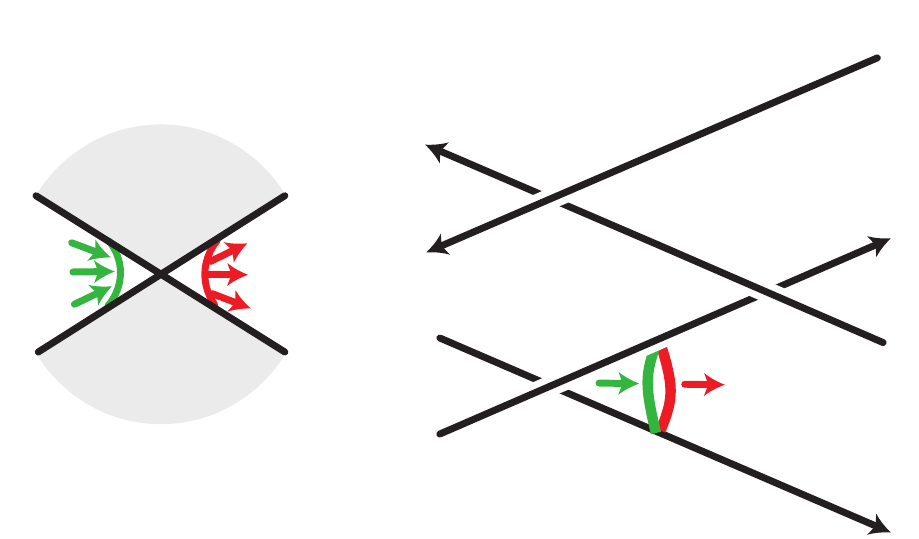}
\end{center}
\caption{The image of a horizontal side of~$\Sn$, depicted here in green, is the corresponding horizontal side of~$\Sb$, depicted here in red.}
\label{horizontalimage}
\end{figure}
Even though the horizontal sides are preserved as a set, they are not preserved individually. 
Indeed, every horizontal side is mapped to the horizontal side that corresponds to the opposed quadrant at the considered double point of~$\gamma$. 
Hence~$\fn$ is an involution of the set of horizontal sides given by a translation with the vector~$(1,1)$ in the square tiles-coordinates. 
This is coherent with the fact that any square is preserved by a $(2,2)$-translation, which in turn is exactly the property of being Ba'cfi-tiled. 
We note that the orientation of the sides is preserved.

\smallskip
\noindent{\bf Example 2:} On Birkhoff-Fried's original surface, the map~$\fn$ acts as the following involution on the set of horizontal edges:~$1 \leftrightarrow12, 2\leftrightarrow 7, 3\leftrightarrow 8, 4\leftrightarrow 9, 5\leftrightarrow10, 6\leftrightarrow11$. This is depicted in Figure~\ref{BFhorizontal}.

\begin{figure}[h]
\begin{center}
\includegraphics[width=.4\textwidth]{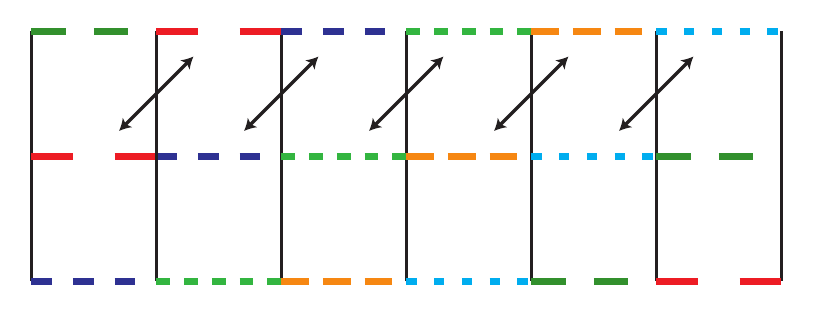}
\end{center}
\caption{The involution of the set of horizontal edges given by a translation with the vector~$(1,1)$ in the square-tile coordinates.}
\label{BFhorizontal}
\end{figure}

We already observed that for every horizontal cylinder, the top and the bottom sides are exchanged by~$\fn$. 
Since the orientation has to be reversed, we deduce that for every horizontal cylinder the transformation is the composition of a transvection and a reflection along a horizontal axis. 

In order to determine the coefficient of the transvection, we need to consider the vertical sides: they will cross the back region toward which they are pointing.
So let~$r$ be a black region in~$\OO\setminus\gamma$ with vertices $P_1,P_2,\dots,P_k$ and suppose for now that it contains no cone point.
Denote by~$a_{P_1}$ the arc formed by the vectors based at~$P_1$ and pointing toward~$r$.
The image $\fn(a_{P_1})$ is then an arc with one extremity at~$P_k$ and the other one at~$P_2$, crossing once all ribbons associated to~$P_2P_3, P_3P_4, \dots, P_{k-1}P_k$. The arc~$\fn(a_{P_1})$ points toward the white regions adjacent to~$r$, 
except near the fibres of the points~$P_3, \dots, P_{k-1}$, where due to continuity the arc~$\fn(a_{P_1})$ needs to point toward the black region opposed to~$r$, see Figure~\ref{verticalimage}.
\begin{figure}[h]
\begin{center}
\includegraphics[width=.5\textwidth]{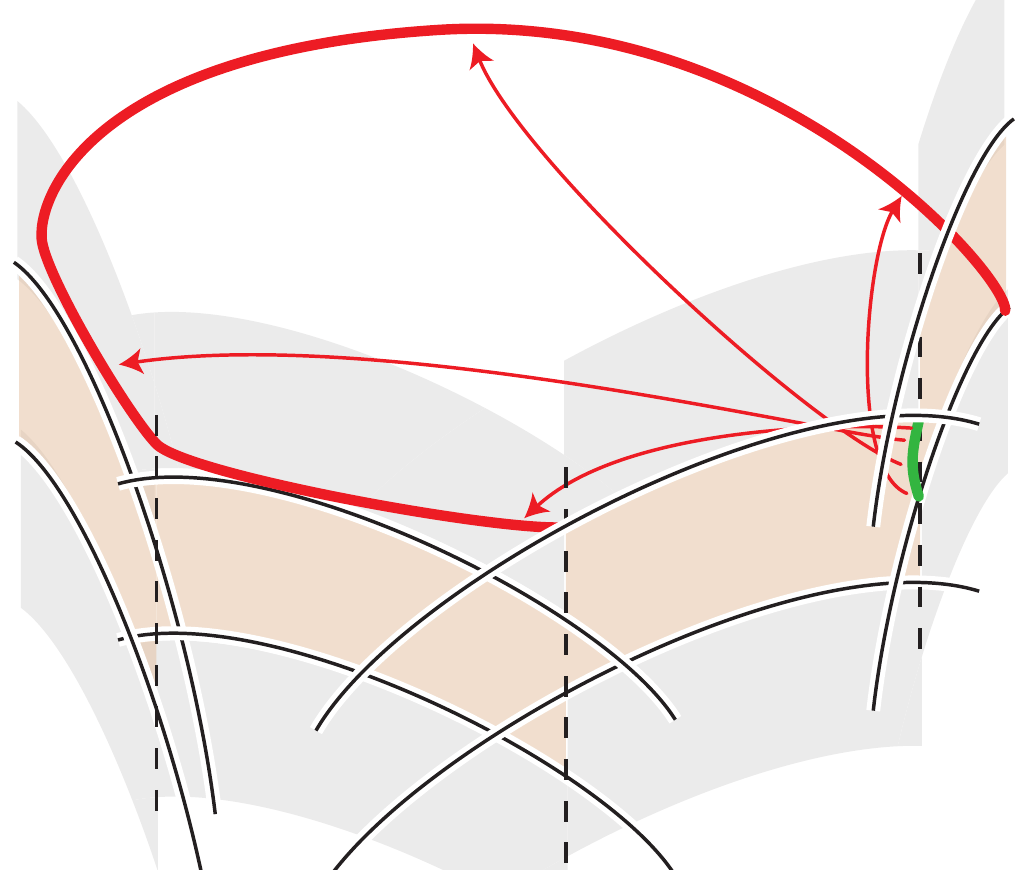}
\end{center}
\caption{The effect of~$\fn$ on an arc that corresponds to a vertical side in the Ba'cfi-tiling.}
\label{verticalimage}
\end{figure}
This implies that the arc~$\fn(a_{P_1})$ cuts the vertical sides of the corresponding square tile, but no horizontal side.
If the cylinder associated to the region~$r$ contains $k$ square tiles, we have that the~$k$ vertical sides are sent on segment of slope~$(k-2,-1)$. 
Therefore~$\fn$ acts on the corresponding cylinder as the horizontal antitwist~$\tau_h^{(1)}$.

Suppose now that the black region~$r$ contains a cone point of order~$j$. The same reasoning, but this time in the $j$-fold cover, shows that the vertical sides are sent on segments of slope~$(jk-2, -1)$. We therefore get that~$\fn$ acts on the corresponding cylinder as the horizontal left-veering antitwist~$\tau_h^{(j)}$.

For the map~$\fb$, we can apply the exact same reasoning, taking the orientation into account. We obtain that vertical sides are sent on vertical sides, and every horizontal side is sent on a segment of slope~$(-1, jk-2)$, so that $\fb$ is a vertical right-veering antitwist.
The monodromy map~$f_\gamma$ is the composition~$\fb\circ \fn$. 
This concludes the proof of Theorem~\ref{T:antitwist}.

\subsection{Examples}
In the cases where~$\Sn$ is of genus one, we write $X=(\begin{smallmatrix}1&1\\ 0&1\end{smallmatrix})$ and $Y=(\begin{smallmatrix}1&0\\ 1&1\end{smallmatrix})$. 
It is known that every matrix in~$\SLZ$ with trace larger than or equal to~$2$ is conjugate to a positive product of $X$s and $Y$s, unique up to cyclic permutation of the letters.

\smallskip
\noindent{\bf Example 1:} $\Sn$ is a torus made of~$2\times2$ square tiles. 
If~$p, q, r, s$ denote the orders of the cone points, where~$p$ and~$r$ correspond to the orders in the black regions,  
then~$\fn$ is the composition of a reflection about a horizontal axis and a horizontal transvection of factor~$2p-2$ on the top cylinder and of factor~$2r-2$ on the bottom line. 
Together this makes a a transvection of factor~$2p+2r-4$, on a width of two blocks.
Therefore, at the global level of~$\Sn$, this is a symmetry about a horizontal axis and a transvection of factor~$p+r-2$.
Similarly,~$\fb$ is the composition of a vertical symmetry and a transvection of factor~$q+s-2$.

After collapsing the boundary components into points, the monodromy is then given by the matrix $(\begin{smallmatrix}-1&0\\q+s-2&1\end{smallmatrix})(\begin{smallmatrix}1&p+r-2\\0&-1\end{smallmatrix})=(\begin{smallmatrix} -1& -(p+r-2)\\ q+s-2 & (p+r-2)(q+s-2)-1 \end{smallmatrix})$, which is conjugate to the product~$X^{p+r-4}YX^{q+s-4}Y$.

\smallskip
\noindent
{\bf Example 2:} 
$\Sn$ is a torus made of 12 square tiles. 
The action of~$\fn$ has a linear part given by $(\begin{smallmatrix}1&4\\0&-1\end{smallmatrix})$.
Concerning~$\fb$, the linear part is similarly given by $(\begin{smallmatrix}-1&0\\4&1\end{smallmatrix})$.
Therefore the linear action of~$f_\gamma$ on the square tiles is given by $(\begin{smallmatrix}-1&0\\4&1\end{smallmatrix})(\begin{smallmatrix}1&4\\0&-1\end{smallmatrix}) =  (\begin{smallmatrix}-1&-4\\4&15\end{smallmatrix})$. 

In order to have the action on~$\Sn$, we need to pass to a basis of~$\Sn$ as a torus, given for example by the vectors $(6,0)$ and~$(2,2)$. 
We then obtain the action $(\begin{smallmatrix}6&2\\0&2\end{smallmatrix})(\begin{smallmatrix}-1&4\\4&15\end{smallmatrix})(\begin{smallmatrix}6&2\\0&2\end{smallmatrix})^{-1}=(\begin{smallmatrix}-5&-8\\12&19\end{smallmatrix})$ on the torus. 

One checks that this matrix has trace 14, which means that there are 12 fixed points. 
These fixed points correspond exactly to the boundary components of~$\Sn$, and there are no others, as was remarked by Ghys~\cite{Ghys}. 
Also the given matrix is conjugated to $X^2YX^2Y$, which is coherent with Hashiguchi's computation~\cite{Hashiguchi}.

\smallskip
\noindent
{\bf Example 3:} $\Sn$ has only one square tile. 
If~$2, q,$ and~$r$ are the order of the cone points, then~$\fn$ acts by~$(\begin{smallmatrix}1&q-2\\0&-1\end{smallmatrix})$, and~$\fb$ acts by~$(\begin{smallmatrix}-1&0\\r-2&1\end{smallmatrix})$. 
We obtain for~$f_\gamma$ the matrix~$(\begin{smallmatrix}-1&2-q\\r-2&(r-2)(q-2)-1\end{smallmatrix})$, which is conjugate to~$X^{q-4}YX^{r-4}Y$. 
This coincides with the result of~\cite{GenusOne}.

\smallskip
\noindent
{\bf Example 4:} $\Sn$ has three square tiles that form two horizontal cylinders but only one vertical cylinder.
The map~$\fn$ acts as a horizontal left-veering antitwist of order~$q-1$ on the cylinder~$ab$, so that the linear part is~$(\begin{smallmatrix} 1&2q-2\\0&-1\end{smallmatrix})$, and it acts as a horizontal left-veering antitwist of order~$r$ on the horizontal cylinder that consists of~$c$ only, hence the linear part is~$(\begin{smallmatrix} 1&r-2\\0&-1\end{smallmatrix})$. 
Since~$ab$ is one cylinder, the action is actually~$(\begin{smallmatrix} 1&q-1\\0&-1\end{smallmatrix})$ in the canonical cylindrical coordinates.
Similarly the map~$\fb$ acts with linear part~$(\begin{smallmatrix} -1&0\\1&1\end{smallmatrix})$ on each square, and so by~$(\begin{smallmatrix} -1&0\\1/3&1\end{smallmatrix})$ on the vertical cylinder~$abc$.

With~$q=3$ and~$r=7$, the orbifold is the triangular orbifold of type~$(2,3,7)$, which is the hyperbolic~$2$-orbifold of minimal area. 
The action of~$f_\gamma$ on the homology~$\HH_1(\Sn)$ can be computed. 
This homology has dimension~$4$, but it is easier to work with a generating set consisting of 5 curves, namely two curves corresponding to the boundary of the white faces, one corresponding to the dark face, and two corresponding to the two double points. 
There, the images of the induced actions of~$\fn$ and~$\fb$ can be described by the following matrices. 
\[
\begin{pmatrix} 
1 & 0 & 0 & 0 & 0\\
0 & 1 & 0 & 0 & 0\\
6 & 7 & 1 &-2 &-2\\
3 & 0 & 0 &-1 & 0\\
3 & 7 & 0 & 0 &-1
\end{pmatrix}
\begin{pmatrix}
1 & 0 & 2 &-2 &-1\\
0 & 1 & 1 & 0 &-1\\
0 & 0 & 1 & 0 & 0\\
0 & 0 & 1 &-1 & 0\\
0 & 0 & 2 & 0 &-1
\end{pmatrix}=\begin{pmatrix}
1 & 0 & 2  &-2 & -1\\
0 & 1 & 1  & 0 & -1\\
6 & 7 & 14 &-10 &-11\\
3 & 0 & 5 & -5 & -3\\
3 & 7 & 11 & -6 & -9
\end{pmatrix}.\]
Then the action on a basis of~$\HH_1(\Sn)$ is obtained by projecting parallely to the kernel of the generating set, and this projection commutes with the action induced 
on~$\HH_1(\Sn)$ by~$f_\gamma$. 
Therefore the characteristic polynomial of the action induced by~$f_\gamma$ is a factor of the characteristic polynomial of the matrix obtained above. 
The characteristic polynomial of the matrix to the right is $(x-1)(x^4-x^3-x^2-x+1)$, so that the characteristic polynomial of the action induced on~$\HH_1(\Sn)$
is~$(x^4-x^3-x^2-x+1)$. We deduce that the stretch factor is approximately $1.722,$ 
and that the map~$f_\gamma$ is actually the minimizer of the stretch factor among genus~2 pseudo-Anosov mapping classes~\cite{LT}.

%

\smallskip
\noindent
{\bf Example 5:} $\fn$ acts on the cylinder $ab$ as $(\begin{smallmatrix} 1&2q-2\\0&-1\end{smallmatrix})$, 
on the cylinder $cdef$ as $(\begin{smallmatrix} 1&2\\0&-1\end{smallmatrix})$, on the cylinder $ghij$ as$(\begin{smallmatrix} 1&2\\0&-1\end{smallmatrix})$, on $klmn$ as $(\begin{smallmatrix} 1&2\\0&-1\end{smallmatrix})$, and on~$o$ as $(\begin{smallmatrix} 1&r-2\\0&-1\end{smallmatrix})$.
For~$\fb$, it acts on the vertical cylinder $abcd$ as $(\begin{smallmatrix} -1&0\\2&1\end{smallmatrix})$, the same for the next ones and on $mno$ it acts as $(\begin{smallmatrix} -1&0\\1&1\end{smallmatrix})$. 

A computation similar to the one in the previous example allows to compute the characteristic polynomial of the action on homology from there. 
For example with $n=3$ which means 11 square tiles, $q=6$ and~$r=16$, one finds the polynomial $x^{12} - 4x^{11} - 4x^{10} - 4x^9 - 4x^8 - 4x^7 - 14x^6 - 4x^5 - 4x^4 - 4x^3 - 4x^2 - 4x + 1$ (see Theorem~\ref{T:noPowerThurstonPennerlater}). 

\section{Constructing an orbifold and the proof of Theorem~\ref{T:converse}}
\label{S:Converse}

In this section, we start with a Ba'cfi-tiled surface~$S$, a left-veering horizontal antitwist~$\tau_h$ and a right-veering vertical one~$\tau_v$, and we construct an orientable 2-orbifold~$\OO$ together with a divide~$\gamma$ so that $\Sn$ is homeomorphic to~$S$ and the divide monodromy~$f_\gamma$ associated to~$\Sn$ is conjugated to the composition~$\tau_v\circ\tau_h$. 
This Section was suggested by Mario Shannon.

\subsection{Constructing a divide from a Ba'cfi-tiled surface}

Let $S$ be a Ba'cfi-tiled surface and let~$Q(S)$ denote the set of its squares. 
The following construction is illustrated on Figure~\ref{F:Intro}, if one goes from the image on the right to the one on the left.

To every square~$s$ in~$Q(S)$ we associate an abstract cooriented edge~$e_s$ (that is, with an orientation of the normal bundle, or, equivalently, with an arrow transverse to it). 
To the right extremity of $e(s)$ we attach the edges $e_{E(s)}, e_{NE(s)},$ and $e_{ENE(s)}$ in trigonometric order. 
To the left extremity we attach the edges $e_{N(s)}, e_{EN(s)},$ and $e_{NEN(s)}$ in clockwise order. 
Thanks to the relation $ENEN(s)=NENE(s)=s$, the above attachments are coherent: edges get attached four by four at all vertices. 
In this way we obtain a 4-valent graph with a coorientation of the edges and a circular order around every vertex.  
We see this as a fat graph. 
Every boundary component of this fat graph corresponds to a cylinder of $S$: those lying on the positive sides of their adjacent edges correspond to horizontal cylinders, while those lying on the negative sides of their adjacent edges correspond to vertical cylinder. 
We then obtain a surface, which we denote by~$\Sigma_S$, by gluing a disc along every boundary component of the fat graph. 
The graph is a divide~$\gamma_S$ on this surface.

\subsection{Adding the cone points from the antitwists}

Assume now that we are given two antitwists~$\tau_h$ and~$\tau_v$ on $S$, such that~$\tau_h$ is a left-veering antitwist along the horizontal cylinders of~$S$ and~$\tau_v$ is a right-veering antitwist along the vertical cylinders. 
For every cylinder $c$, we denote by $n(c)$ the exponent of~$c$, that is, the number of times~$\tau_h$ ({\it resp.}~$\tau_v$) wraps around $c$. 
For every cylinder, we then add a cone point of order~$n(c){-}1$ in the corresponding face of~$\Sigma_S$, thus obtaining an orbifold~$\OO_S$ with a divide~$\gamma_S$.

\begin{proof}[Proof of Theorem~\ref{T:converse}]
Starting from a square-tiled surface~$S$ and a composition of a horizontal left-veering antitwist~$\tau_h$ with a vertical right-veering one~$\tau_v$, we construct the orbifold~$\OO_S$ and the divide~$\gamma_S$ as above. 
Everything has been done so that the associated divide monodromy~$f_{\gamma_S}$ is $\tau_v\circ\tau_h$.
\end{proof}


\section{Stretch factors and comparisons with other constructions}\label{S:stretch factor}
 
We study here the stretch factors of divide monodromies. 
Their locations in~$\C$ or the properties of their Galois conjugates allow to compare divide monodromies with other known construction of pseudo-Anosov maps. 

\subsection{Stretch factors of divide monodromies}
Divide monodromies tend to have rather small stretch factor. 
For example the case~$q=3, r=7$ in Example~3 yields a map~$f_\gamma$ on a torus that is of trace 3, hence of stretch factor~$\frac{3+\sqrt 5}{2}$ --- the smallest possible value for an Anosov map with orientable invariant foliations in genus~1.

Similarly the same case~$q=3, r=7$ in Example~4 yields a map~$f_\gamma$ on a genus~$2$ surface which happens to coincide with the Lanneau-Thiffeault map~\cite{LT}, so it also has the smallest possible stretch factor among pseudo-Anosov maps with orientable invariant foliations in genus~2. 

The next statement partly explains these facts. 
Recall that given a pseudo-Anosov map~$f:S\to S$, one forms the 3-manifold~$M_f:=S\times[0,1]/\!_{(x,1)\sim(f(x),0)}$. 
The suspension flow~$\partial_z$, where~$z$ in the second coordinate is then of pseudo-Anosov type. 
The associated fibred cone~$\mathcal C_f$ is the set of classes~$\sigma\in\HH_2(M_f, \Q)$ such that~$\sigma$ contains a multiple of a global section for~$\partial_z$. 
By a theorem of Schwartzman~\cite{Schwartzman},~$\mathcal C_f$ is an open cone. 
Given a class~$\sigma\in\mathcal C_f$ and a global section~$S\in\sigma$, the induced first-return map~$f_S:S\to S$ is of pseudo-Anosov type, hence it has a stretch factor~$d_S$. 
Fried remarked that the function~$d_S^{|\chi(S)|}$, called the {normalized stretch factor}, is convex and tends to infinity when $\sigma$ approaches the boundary of~$\mathcal C_f$, hence it has a global minimum inside~$\mathcal C_f$.
We now prove Theorem~\ref{T:stretchfactor}, which we restate below for convenience.

\begin{theorem}\label{T:stretchfactorlater}
Let~$S$ be a Ba'cfi-surface and let~$f:S\to S$ be a divide monodromy. 
Denote by~$S^o$ the surface obtained from~$S$ by removing the vertices of the square tiles and by~$f^o$ the induced map. 
Then~$f^o$ minimizes the normalized stretch factor in~$\mathcal C_{f^o}$. 
\end{theorem}

\begin{proof}
Let $\OO$ and $\gamma$ be associated to~$S$ and $f$ by Theorem~\ref{T:converse}. 
A section of the fibred cone associated to~$S^o$ was computed in~\cite{Intersection}: it is the dual unit ball of the intersection norm associated to~$\gamma$ on~$\OO$. 
This section happens to be symmetric, and the center of symmetry corresponds to the class~$[S^o]$ (actually, the fact that the cone is symmetric with respect to~$[S^o]$ can be seen directly). 
By symmetry and uniqueness of the minimizer, $[S^o]$ gives the minimum of the normalized stretch factor.
\end{proof}

We now prove Theorem~\ref{T:diameter}, which we restate below for convenience.

\begin{theorem}\label{T:diameterlater}
Given a hyperbolic good 2-orbifold~$\OO$ and a geodesic divide~$\gamma$ on~$\OO$ such that all regions of~$\OO\setminus\gamma$ have diameter at most~$D$, the stretch factor of the associated map~$f_\gamma$ is bounded from above by~$e^{2D}$.
\end{theorem}

\begin{proof}
As mentioned earlier, in this setting, the monodromy is given by the first-return map along the geodesic flow. 
This flow is of Anosov type with expansion rate~$1$. 
Therefore the log of the stretch factor of the first-return map is the average (with respect to the measure of maximal entropy) of the first-return time. 
The latter being bounded by twice the maximal diameter of a face in~$\OO\setminus\gamma$, the result follows.
\end{proof}

\begin{corollary}\label{T:comparaison}
There exist sequences of Ba'cfi surfaces~$({\Sn}^{(n)})_{n\in\N}$ and divide monodromies~$(f_\gamma^{(n)})$ living in the strata $\mathcal{H}_{non-hyp}(2g-2)$ and~$\mathcal{H}_{non-hyp}(g-1,g-1)$ respectively, whose associated stretch factors~$(\lambda^{(n)})_{n\in\N}$ tend to~1. 
\end{corollary}

\begin{proof}
Fixing a hyperbolic surface~$S$, it is known that generic geodesics tend to equidistribute with respect to the Liouville measure~\cite{Bowen}.
Thus a long enough generic geodesic will visit all balls of radius at least~$\epsilon$ in~$\U\OO$.  
Using the Anosov Closing Lemma~\cite[p.270]{Hasselblatt}, one can close such a geodesic without changing the fact that it visits all $\epsilon$-balls. 

Now if a region of~$\OO\setminus\gamma$ has diameter at least $D$, then there is a ball of diameter~$D/3$ avoided by~$\U\gamma$ in~$\U\OO$. 
Indeed, one can consider the ball formed of those vectors that are based in a circle whose diameter is the middle third of the diagonal and which are orthonal to this diameter up to an angle~$1/D$.
\begin{center}
\includegraphics[width=.3\textwidth]{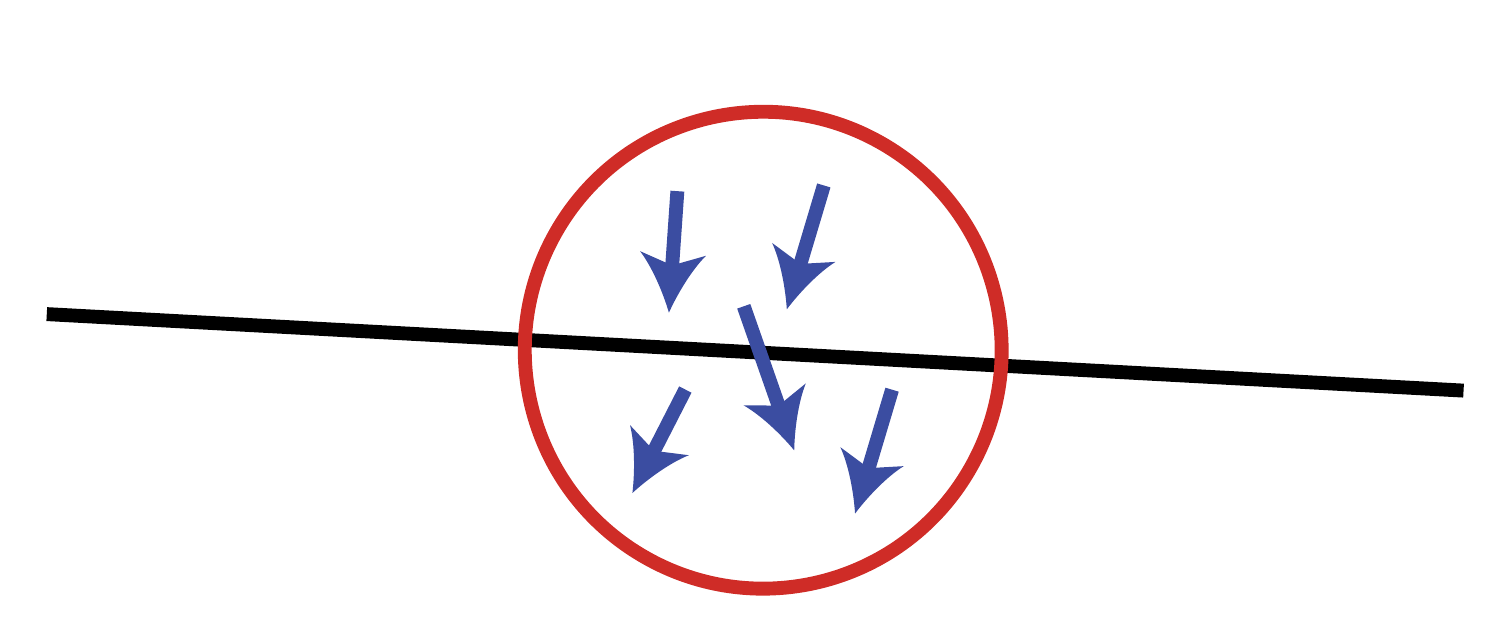}
\end{center}
Therefore a long enough generic closed geodesic will leave no region in its complement with diameter larger than~$\epsilon$. 
By Corollary~\ref{T:comparaison} the associated expansion rate is bounded by~$e^{2\epsilon}$. 

Concerning the strata in which these maps live, we first notice that they have two singularities only, since singularities only arise at the boundary of~$\Sn$ (the interior of~$\Sn$ being transverse to the weak foliations of the geodesic flow). The boundary of~$\Sn$ consists of two curves if~$\gamma$ consists of one curve only. 
Therefore the above maps live in~$\mathcal{H}(g-1,g-1)$. 
By a result of Boissy and Lanneau~\cite{BL}, the stretch factor is uniformly bounded away from 1 in the hyperelliptic component, so our examples have to live in the other connected component of the stratum. 

In order to go to~$\mathcal{H}(g-2)$, one can work with an orbifold having one cone point of order 2 and consider geodesic divides~$\gamma$ going through this order 2 point. 
One mimics the argument, but this time~$\Sn$ has only one boundary component. 
\end{proof}

\subsection{A comparison with Thurston's and Penner's construction}
Penner and Thurston gave machineries to construct pseudo-Anosov maps in terms of Dehn twists or multitwists~\cite{Penner, ThurstonConstruction}. 
It is therefore a natural question to ask about the intersection of those constructions with the divide monodromies. 
In general the constructions of Thurston and Penner yield pseudo-Anosovs whose invariant foliations may or may not be orientable. 
Since our divide monodromies always have orientable foliations, our construction is certainly not a generalization of the constructions of Thurston or Penner. 
Reversely, we can also construct divide monodromies that are pseudo-Anosov but do not arise from Thurston's or Penner's construction. 

\begin{corollary}\label{T:noThurstonPenner}
In the hyperelliptic strata $\mathcal{H}_{ell}(2g-2)$ and~$\mathcal{H}_{ell}(g-1,g-1)$ there are maps not arising from Thurston's or Penner's construction.
\end{corollary}

\begin{proof}
It is known that the stretch factors of maps arising from Penner's or Thurston's construction are uniformly bounded away from 1~\cite{Leininger}. 
Hence the sequences of Corollary~\ref{T:comparaison} cannot be obtained from Penner's or Thurston's construction.
\end{proof}

Our goal now is to exhibit divide monodromies that are pseudo-Anosov but none of whose powers can be obtained by Thurston's or Penner's construction.
We use the following proposition to rule out that a pseudo-Anosov map can have a power in Thurston's or Penner's construction.

\begin{proposition}
\label{excludeTP}
Let~$\lambda$ be the stretch factor of a pseudo-Anosov map~$f$.
\begin{enumerate}
\item If a Galois conjugate of~$\lambda$ lies outside~$\R\cup\mathbb{S}^1$, then no power~$f^k$ arises from Thurston's construction,\label{item1}
\item if a Galois conjugate of~$\lambda$ lies on~$\mathbb{S}^1$, then no power~$f^k$ arises from Penner's construction. \label{item2}
\end{enumerate}
\end{proposition}

\begin{proof}
It follows from the proof of Lemma~$8.2$ in Strenner's article~\cite{Strenner} that~$\mathbb{Q}(\lambda) = \mathbb{Q}(\lambda^k)$
and that if~$\lambda_1,\dots,\lambda_l$ are the Galois conjugates of~$\lambda$, then the Galois conjugates of~$\lambda^k$ are 
exactly~$\lambda_1^k,\dots,\lambda_l^k$. 

For~(\ref{item1}), note that if a Galois conjugate~$\lambda_i$ of~$\lambda$ lies outside~$\R\cup\mathbb{S}^1$, then
also~$\lambda_i^k$ lies outside~$\R\cup\mathbb{S}^1$. Indeed, the only way for~$\lambda_i^k$ to lie in~$\R\cup\mathbb{S}^1$
is if the argument of~$\lambda_i$ is a rational multiple of~$\pi$. If this were the case, we would have~$\lambda_i^{k} = \overline{\lambda_i}^{k}\in\R$. 
Since an irreducible integer polynomial has no multiple zeroes, this implies that~$\mathrm{deg}(\lambda^k)<\mathrm{deg}(\lambda)$, 
a contradiction. Therefore,~$\lambda_i^k$ lies outside~$\R\cup\mathbb{S}^1$. As the stretch factor of~$f^k$ equals~$\lambda^k$, the statement now 
follows from the fact that all Galois conjugates of the stretch factor of a pseudo-Anosov mapping class arising from Thurston's construction lie in~$\R\cup\mathbb{S}^1$, 
a result due to Hubert and Lanneau~\cite{HL}. 

Similarly, for~(\ref{item2}), we note that if a Galois conjugate~$\lambda_i$ of~$\lambda$ lies on the unit circle, then so does~$\lambda_i^k$. 
The statement then follows directly from the fact that a pseudo-Anosov stretch factor with a Galois conjugate on the unit circle cannot arise from Penner's construction,
a result due to Shin and Strenner~\cite{SS}. 
\end{proof}

We are now able to prove Theorem~\ref{T:noPowerThurstonPenner}, which we restate below for convenience.

\begin{theorem}
\label{T:noPowerThurstonPennerlater}
There exists a divide monodromy~$f$ so that no power~$f^k$ is obtained from Thurston's or Penner's construction, where~$k\ge1$.
\end{theorem}

\begin{proof}
The stretch factor~$\lambda$ of the divide monodromy~$f$ of the even ping-pong with 11 squares on the~$(2,6,16)$-orbifold has minimal polynomial
$$x^{12} - 4x^{11} - 4x^{10} - 4x^9 - 4x^8 - 4x^7 - 14x^6 - 4x^5 - 4x^4 - 4x^3 - 4x^2 - 4x + 1.$$
This polynomial has four real roots, four roots on~$\mathbb{S}^1$ and four roots outside~$\R\cup\mathbb{S}^1$. In particular, no power~$f^k$ 
can be obtained from Thurston's or Penner's construction by Proposition~\ref{excludeTP}. 
\end{proof}


\section{Looking for an invariant train track}\label{S:TrainTrack}

In the case where divide monodromies are pseudo-Anosov, one would like to construct an invariant train track. 
Unfortunately, we are not able to achieve this goal in general: we can only do it if~$\OO\setminus\gamma$ consists of $n$-gons with $n\ge 5$. 
In other words, we can describe an invariant train track for a divide monodromy only if the cylinders are long enough.

Given a divide~$\gamma$ on an orbifold~$\OO$, 
we construct a train track~$T_\gamma$ on~$\Sn$ as follows: 
every $e$ edge of~$\gamma$ is naturally cooriented from the white face to black face adjacent to it. 
We then put two vertices $v_l(e)$ and $v_r(e)$, near the left (\emph{resp.}\ right) end of $e$. 
We place one edge $r(e)$, called \emph{red}, between these two vertices. 
Next for every white face $f^\circ$ and every vertex $v$ adjacent to it, denote by $e_1, e_2$ the two edges 
adjacent to $v$ and $f^\circ$ (in clockwise order around~$f^\circ$), and place an edge $g_{v, f^\circ}$ connecting $v_r(e_1)$ to $v_l(e_2)$, called \emph{green}. 
Finally for every black face $f^\bullet$ and every vertex $v$ adjacent to it, denote by $e_1, e_2$ the two edges 
adjacent to $v$ and $f^\bullet$ (in clockwise order again), and place an edge $b_{v, f^\bullet}$ connecting $v_l(e_1)$ to $v_r(e_2)$, called \emph{blue}. 
A local example of this train track is depicted in Figure~\ref{tt_ex} below.

\begin{figure}[h]
\begin{center}
\includegraphics[width=.3\textwidth]{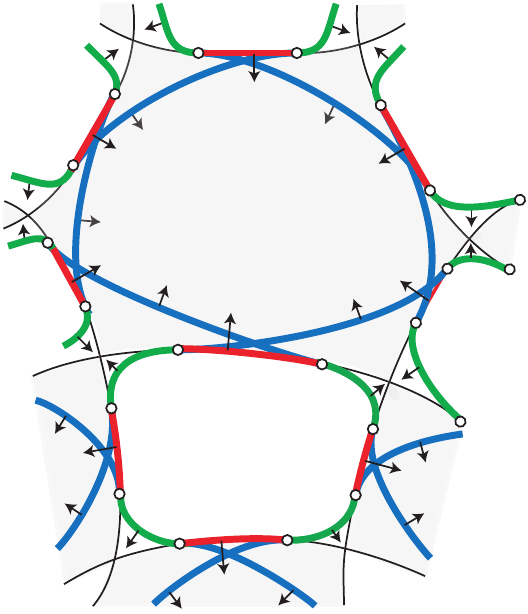}
\hspace{5mm}
\includegraphics[width=.3\textwidth]{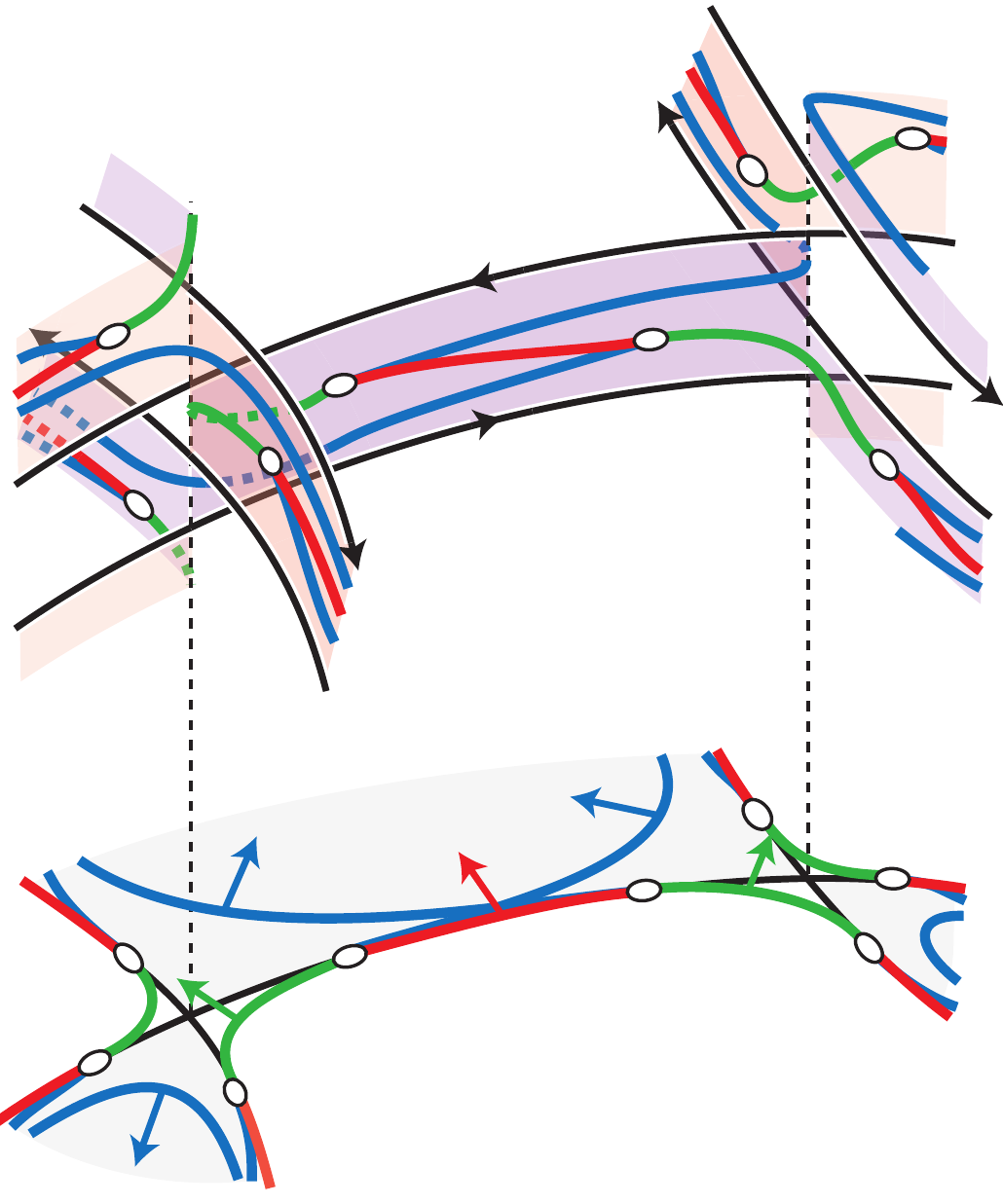}
\caption{A local picture of the train track~$T_\gamma$ on~$\Sn$, and how it sits in the unit tangent bundle.} 
\label{tt_ex}
\end{center}
\end{figure}
The construction of the train track actually works for any divide, but only if the all the~$n$-gons have~$n\ge 5$ are we able to see that the divide monodromy 
induces a train track map. So let us turn to the first-return map. 
As before, we only consider the half-first-return map~$\fn:\Sn\to\Sb$. 
So we take the brother train track on~$\Sb$ constructed in the exact same way. 
The images of the three type of edges are depicted in Figure~\ref{tt_images} below. 

\begin{figure}[h]
\begin{center}
\includegraphics[width=.9\textwidth]{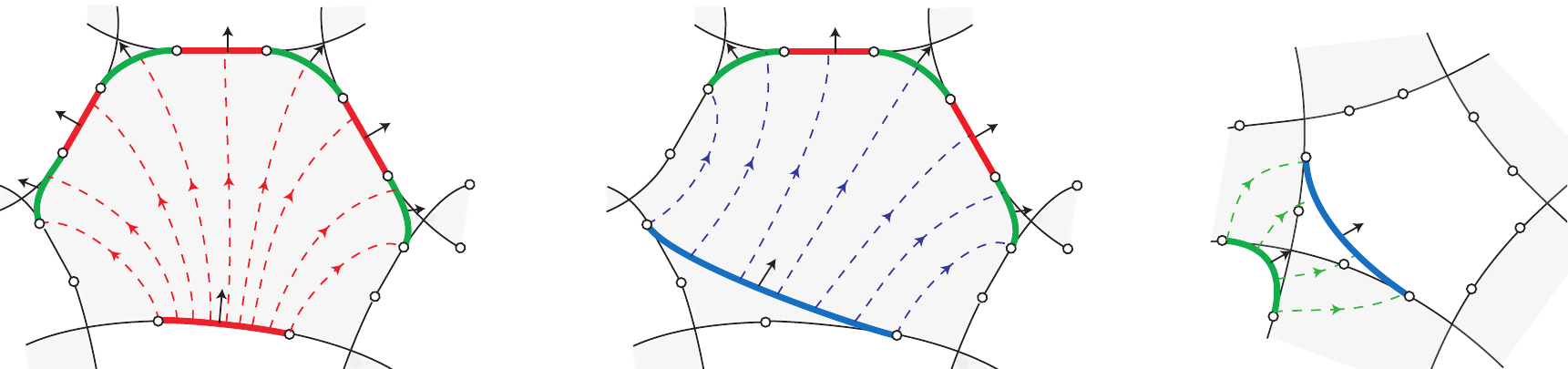}
\caption{The images of the three types of train track edges under the map~$\fn$.}
\label{tt_images}
\end{center}
\end{figure}

With this definition, one can see that~$\fb\circ\fn$ indeed induces a train-track map on~$T_\gamma$. 
However, the way we defined the induced map of~$\fn$, it does not respect the orientations in 3- or 4-gons,
so our train track is of limited practicality. 
We note that as soon as we have defined the train track~$T_\gamma$, we can deduce that a divide monodromy is pseudo-Anosov without 
referring to a geometric result on geodesic flows, but using only the combinatorics of train tracks.
We summarize the discussion of this section by proposing the following construction of pseudo-Anosov maps.
In the introduction, this construction was stated as Theorem~\ref{Ba'cfi-pa_thm}.

\begin{theorem}[Ba'cfi-construction of pseudo-Anosov mapping classes]
\label{Ba'cfi-palater}
Let~$S$ be a Ba'cfi-tiled surface so that all cylinders have width at least 5. Furthermore, let~$\tau_h$ a left-veering horizontal antitwist and let~$\tau_v$ be a right-veering vertical antitwist. Then~$\tau_v\circ\tau_h$ is pseudo-Anosov with stretch factor at least~$\frac{5}{2}$.
\end{theorem}

\begin{proof}
We already constructed an invariant train track for~$\tau_v\circ\tau_h$ in our discussion of this section. 
Let~$C$ be the cone of positive weights on the edges of the train track that satisfy the switch condition, that is, 
at each switch of the train track, the weights of the incoming edges sum up to the weights of the outgoing edges. 
These are exactly the weights that define a singular measured foliation, compare, for example, 
with Penner's construction of pseudo-Anosov mapping classes~\cite{Penner}. 
It follows directly from the definition that the action~$M$ induced by the map~$\tau_v\circ\tau_h$ on the cone of positive edge weights
preserves~$C$. We claim that the dominant eigenvalue of the restriction of~$M$ to~$C$ is bounded from below by~$\frac{5}{2}$. 
The eigenvector corresponding to this eigenvalue defines an invariant singular measured foliation for~$\tau_v\circ\tau_h$.  
This implies that~$\tau_v\circ\tau_h$ is pseudo-Anosov with stretch factor bounded from below by~$\frac{5}{2}$.

In order to prove the claim, let~$w$ be the weight with entries~$2$ for the green edges and and entries~$1$ for the red and the blue edges. 
It is immediate that~$w\in C$. Note that the induced actions of both~$\tau_h$ and~$\tau_v$ send a green edge to a blue edge, 
a blue edge to at least one red edge and two green edges, and a red edge to at least three green edges and two red edges, compare with Figure~\ref{tt_images}.
Note that here we use that all cylinders have size at least 5, or, in other words, that all the complementary regions of the divide are~$n$-gons with~$n\ge5$.
It follows for the action~$M$ induced by the composition~$\tau_v\circ\tau_h$ that~$Mw\ge \frac{5}{2}w$. 
In particular, the dominant eigenvalue of the restriction of~$M$ to~$C$ is at least~$\frac{5}{2}$ by Perron-Frobenius theory, which finishes the proof.
\end{proof}

\section{Questions}\label{S:Questions}

We raise several questions concerning divide monodromies. 

\begin{question}
How to detect on a Ba'cfi-tiled surface and a product of two antitwists whether they come from a geodesic divide on a hyperbolic orbifold? 
\end{question}

This question has a theoretical answer: one can form a divide~$\gamma$ from a Ba'cfi-tiled surface, and the question is whether this divide can be made geodesic on a hyperbolic orbifold. 
It is known~\cite{HassScott} that this is the case if the divide contains no immersed 1-gon or 2-gon, and there are algorithms that detect this. 
However, we would like to see whether there is a simple criterion directly on the Ba'cfi-tiled surface that shortcuts these remarks.

If a divide can be made geodesic on a hyperbolic surface, then the associated divide monodromy is pseudo-Anosov. 
More generally, we ask:

\begin{question}
How to detect when the product of two antitwists is pseudo-Anosov?
\end{question}

The same remark applies here, with the additional complication that even when the orbifold is not hyperbolic, the map~$f_\gamma$ may be pseudo-Anosov. 
Theorem~\ref{Ba'cfi-pa_thm} gives a partial answer to this question. 
We wonder whether the hypotheses in Theorem~\ref{Ba'cfi-pa_thm} could be weakened: 

\begin{question}
Suppose that $\gamma$ is a geodesic divide on a hyperbolic orbifold~$\OO$, is there a geometric way to construct an invariant train track, even when $\OO\setminus\gamma$ has triangles or quadrangles?
\end{question}

Concerning stretch factors, Theorem~\ref{T:diameter} gives upper bounds. Can we similarly find lower bounds?

\begin{question}
Suppose that $\gamma$ is a geodesic divide on a hyperbolic orbifold~$\OO$, are there geometric lower bounds on the stretch factor of~$f_\gamma$?
\end{question}

One could think of comparisons with the volume of the complements, see for examples the lower bounds given by Rogriguez-Migueles~\cite{Rodriguez}.

As remarked in Theorem~\ref{T:stretchfactor}, a divide monodromy map always corresponds to the center of the associated fibred face. 
We wonder whether the converse is true or not:

\begin{question}
Assume that $f$ is mapping class whose associated fibred face is symmetric and such that~$f$ corresponds to the center of the fibred face. Does that imply that $f$ is a divide monodromy?
\end{question}

\bigskip

\bibliographystyle{siam}

\end{document}